\documentclass[a4paper]{article}
\usepackage{hyperref}
\usepackage{stmaryrd}
\usepackage{amsmath,bm}      
\usepackage{amsfonts}
\usepackage{amsthm}
\usepackage{amssymb}
\usepackage{mathrsfs}
\usepackage{tensor}
\usepackage{tikz}
\usepackage{graphicx}
\usepackage{adjustbox}
\makeatletter
\newcommand*\bigcdot{\mathpalette\bigcdot@{.5}}
\newcommand*\bigcdot@[2]{\mathbin{\vcenter{\hbox{\scalebox{#2}{$\m@th#1\bullet$}}}}}
\makeatother
\usetikzlibrary{matrix, arrows}
\usetikzlibrary{positioning}
\usetikzlibrary{decorations.pathreplacing}
\theoremstyle{plain}
\newtheorem*{Example*}{\bfseries{\emph{Example}}}
\newtheorem*{Notation*}{\bfseries{\emph{Notation}}}
\newtheorem*{Theorem*}{\bfseries{\emph{Theorem}}}
\newtheorem*{Lemma*}{\bfseries{\emph{Lemma}}}
\newtheorem*{Proposition*}{\bfseries{\emph{Proposition}}}
\newtheorem*{Corollary*}{\bfseries{\emph{Corollary}}}
\newtheorem*{Remark*}{\bfseries{\emph{Remark}}}
\newtheorem*{Remarks*}{\bfseries{\emph{Remarks}}}
\newtheorem*{Def*}{\bfseries{\emph{Definition}}}
\newtheorem*{Conjecture*}{\bfseries{\emph{Conjecture}}}
\newtheorem*{sketch proof*}{Sketch proof}
\newtheorem{Theorem}{\bfseries{\emph{Theorem}}}[section]

\newtheorem*{Examples*}{\bfseries{\emph{Examples}}}

\newtheorem{Lemma}[Theorem]{\bfseries{\emph{Lemma}}}
\newtheorem{Proposition}[Theorem]{\bfseries{\emph{Proposition}}}
\newtheorem{Corollary}[Theorem]{\bfseries{\emph{Corollary}}}

\newcommand{\fil}{\text{Fil}}

\newcommand{\logmw}{\text{log-MW}}

\newcommand{\rig}{\text{rig}}
\newcommand{\logcris}{\text{log-cris}}
\newcommand{\logrig}{\text{log-rig}}

\newcommand{\spec}{\text{Spec }}
\newcommand{\specno}{\text{Spec}}
\newcommand{\spf}{\text{Spf }}

\newcommand{\spwf}{\text{Spwf }}

 \title{Overconvergent de Rham-Witt cohomology for semistable varieties}
 \date{January 24, 2019}
  \author{Oliver Gregory\footnote{The first named author is supported by the ERC Consolidator Grant 681838 “K3CRYSTAL”.} \ and Andreas Langer}
\begin{document}
\maketitle
\begin{abstract}
We define an overconvergent version of the Hyodo-Kato complex for semistable varieties $Y$ over perfect fields of positive characteristic, and prove that its hypercohomology tensored with $\mathbb{Q}$ recovers the log-rigid cohomology when $Y$ is quasi-projective. We then describe the monodromy operator using the overconvergent Hyodo-Kato complex. Finally, we show that overconvergent Hyodo-Kato cohomology agrees with log-crystalline cohomology in the projective semistable case.
\end{abstract}

\section{Introduction}
For a proper and smooth variety $Y$ over a perfect field $k$ of characteristic $p>0$, the hypercohomology of the Deligne-Illusie de Rham-Witt complex $W\Omega_{Y/k}^{\bullet}$ tensored with $\mathbb{Q}$ computes - using the comparison isomorphism with crystalline cohomology \cite{Ber97} - the rigid cohomology
\begin{equation*}
H^{\ast}_{\rig}(Y/W(k)[1/p])\cong\mathbb{H}^{\ast}(Y,W\Omega_{Y/k}^{\bullet})\otimes\mathbb{Q}
\end{equation*}
However, rigid cohomology is well-defined without any properness assumption on $Y$. In \cite{DLZ11}, Davis-Langer-Zink define an overconvergent de Rham-Witt complex $W^{\dagger}\Omega_{Y/k}^{\bullet}$, which is a subcomplex of $W\Omega_{Y/k}^{\bullet}$, and show that 
\begin{equation*}
H^{\ast}_{\rig}(Y/W(k)[1/p])\cong\mathbb{H}^{\ast}(Y,W^{\dagger}\Omega_{Y/k}^{\bullet})\otimes\mathbb{Q}
\end{equation*}
for $Y$ quasi-projective and smooth over $k$. 

On the other hand, one could instead relax the smoothness condition on $Y$. Let $S_{0}=(\spec k,\mathbb{N})$ be the standard log point, and let $Y$ be a fine $S_{0}$-log scheme. Let $\mathfrak{S}_{0}$ be the (weak) formal log scheme $(\spf W,\mathbb{N}\rightarrow W,1\mapsto 0)$. Then Grosse-Kl\"{o}nne \cite{Gro05} defines the log-rigid cohomology $H_{\logrig}^{\ast}(Y/\mathfrak{S}_{0})$ of $Y$ (we recall the definition in \S\ref{log-rigid section}). Grosse-Kl\"{o}nne shows that the log-rigid cohomology of $Y$ agrees with Shiho's log-convergent cohomology of $Y$ whenever $Y$ is a semistable variety whose irreducible components are proper. In particular, by the comparison between log-convergent and log-crystalline cohomology \cite{Shi02}, there is an isomorphism 
\begin{equation*}
H_{\logrig}^{\ast}(Y/\mathfrak{S}_{0})\cong\mathbb{H}^{\ast}(Y,W\omega_{Y/k}^{\bullet})\otimes\mathbb{Q}
\end{equation*}
for proper semistable varieties $Y$ over $k$, where $W\omega_{Y/k}^{\bullet}$ is the Hyodo-Kato complex \cite{HK94}. 

There is, however, currently no Hyodo-Kato style theory available for non-proper semistable varieties. We shall define a complex $W^{\dagger}\omega_{Y/k}^{\bullet}$, functorial in $Y$, which computes log-rigid cohomology in non-proper situations. More precisely, we prove
\begin{Theorem}\label{intro theorem}
Let $Y$ be a quasi-projective semistable variety over $S_{0}$. Then there is a quasi-isomorphism \begin{equation*}
R\Gamma_{\emph{\logrig}}(Y/\mathfrak{S}_{0})\cong\mathbb{R}\Gamma(Y, W^{\dagger}\omega_{Y/k}^{\bullet}\otimes\mathbb{Q})
\end{equation*}
\end{Theorem}
We then describe the monodromy operator on log-rigid cohomology in terms of the overconvergent Hyodo-Kato complex, using the method of \cite{Mok93}.

Using the overconvergent Hyodo-Kato complex we can formulate Grosse-Kl\"{o}nne's Hyodo-Kato isomorphism \cite{Gro05} in the non-proper case as follows: For a strictly semistable weak formal scheme $\mathfrak{X}$ over $\spwf W(k)$ with generic fibre $X$, a $K_{0}=W(k)[1/p]$-dagger space, and closed fibre $Y$ which we assume to be quasi-projective, there is a canonical isomorphism between the de Rham cohomology of $X$ and the (rational) overconvergent Hyodo-Kato cohomology of $Y$ (\cite{Gro05}, Corollary 3.5).

In \cite{CDN18}, log-rigid cohomology and overconvergent syntomic cohomology are used to study the $p$-adic cohomology of semistable $p$-adic Stein spaces, notably Drinfeld half-spaces,  within the context of a hoped for $p$-adic local Langlands correspondence. We expect that the overconvergent Hyodo-Kato complex will have interesting applications in this area.

Note that even in the smooth case $X/k$, \cite{DLZ11} only construct a map $R\Gamma_{\rig}(X)\rightarrow\mathbb{R}\Gamma(X,W^{\dagger}\Omega_{X/k}^{\bullet}\otimes\mathbb{Q})$ when $X$ is quasi-projective. We will see in the proof of Theorem \ref{main comparison} (= Theorem \ref{intro theorem}) in Section \ref{main comparison section} why we need to assume quasi-projectivity. In a recent preprint \cite{Law18}, Lawless has been able to remove the quasi-projectivity hypothesis at least in the smooth case. For details we refer to \cite{Law18}. In work in progress, Lawless also intends to remove this hypothesis in the semistable case and hence we expect Theorem \ref{intro theorem} to hold unconditionally as well. 

In the final part of the paper, we compare the usual and overconvergent Hyodo-Kato cohomology in the proper case:

\begin{Theorem}
Let $Y$ be a projective semistable variety over $S_{0}$. Then the canonical map 
\begin{equation*}
\mathbb{H}^{\ast}(Y,W^{\dagger}\omega_{Y/k}^{\bullet})\rightarrow\mathbb{H}^{\ast}(Y,W\omega_{Y/k}^{\bullet})= H_{\emph{\logcris}}^{\ast}((Y,M)/(W(k),W(L)))
\end{equation*}
induced by the inclusion $W^{\dagger}\omega_{Y/k}^{\bullet}\subset W\omega_{Y/k}^{\bullet}$ is an isomorphism of finite type $W(k)$-modules. Here $M$ is the log structure on $Y$ given by $\mathcal{O}_{Y}\cap u_{\ast}\mathcal{O}_{U}^{\times}$ where $u:U\hookrightarrow Y$ is a smooth dense open, and $W(L)$ is the canonical lifting of the log structure $L$ on $\spec k$ given by $1\mapsto 0$ (previously denoted by $S_{0}$). 
\end{Theorem}

It should be noted that in fact we give two definitions of the overconvergent Hyodo-Kato complex in this paper. The first is constructed in the style of \cite{HK94}, and the second in the more modern approach of \cite{Mat17}. We prove that the two complexes are the same. Along the way, this has the serendipitous consequence that we show that Matsuue's log de Rham-Witt complex $W\Lambda_{Y/(R,\mathbb{N})}^{\bullet}$ (see \S\ref{Matsuue} for the notation) gives the Hyodo-Kato complex $W\omega_{Y/k}^{\bullet}$ in the special case that $R=k$, thus filling a gap in the literature.  

We assume that the reader is familiar with the de Rham-Witt complex of Deligne-Illusie and its basic properties, including its explicit description for (Laurent)-polynomial algebras \cite{LZ04}, \'{e}tale base change results, and the overconvergent version in the smooth case proven in \cite{DLZ11}.
\\
\\
\noindent {\bf{Acknowledgements.}} The first named author expresses his thanks to Elmar Grosse-Kl\"{o}nne for a very helpful conversation. Both authors thank the referee for helpful comments and improvements to the paper.

\section{Log-rigid cohomology}\label{log-rigid section}

Let $k$ be a field of characteristic $p>0$, let $W=W(k)$ be the Witt vectors of $k$ and $K:=\text{Frac }W$. Set $W_{n}=W/p^{n}$ for $n\in\mathbb{N}$. We will write $\mathfrak{S}_{0}$ for the (weak) formal log scheme $(\spf W,\mathbb{N}\rightarrow W,1\mapsto 0)$. The special fibre of $\mathfrak{S}_{0}$ is the standard log point $S_{0}=(\spec k,\mathbb{N}\rightarrow k,1\mapsto 0)$.

We briefy recall Grosse-Kl\"{o}nne's definition of the log-rigid cohomology of a fine $S_{0}$-log scheme $Y$. For the details, one should consult \cite{Gro05}. Let $Y=\bigcup_{i\in I}V_{i}$ be an open covering, and suppose there are exact closed immersions $V_{i}\hookrightarrow\mathfrak{P}_{i}$ into log smooth weak formal $\mathfrak{S}_{0}$-log schemes for each $i$. For each $H\subset I$, choose an exactification of the diagonal embedding:
\begin{equation*}
V_{H}:=\bigcap_{i\in H}V_{i}\stackrel{\iota}{\hookrightarrow}\mathfrak{P}_{H}\xrightarrow{f}\prod_{i\in H}\mathfrak{P}_{i}
\end{equation*}
(this means that $\iota$ is an exact closed immersion and $f$ is log-\'{e}tale). Then the log de Rham complex $\omega_{\mathfrak{P}_{H}/\mathfrak{S}_{0}}^{\bullet}$ tensored with $\mathbb{Q}$ induces a complex of sheaves $\omega_{\mathfrak{P}_{H,K}}^{\bullet}$ on the $K$-dagger space $\mathfrak{P}_{H,K}$ associated to the generic fibre of $\mathfrak{P}_{H}$ (see \cite{Gro00}). Let $\text{sp}:\mathfrak{P}_{H,K}\rightarrow\mathfrak{P}_{H}$ be the specialisation map, and write $]V_{H}[_{\mathfrak{P}_{H}}^{\dagger}:=\text{sp}^{-1}(V_{H})$ for the tubular neighbourhood of $V_{H}$ in $\mathfrak{P}_{H,K}$. Then $]V_{H}[_{\mathfrak{P}_{H}}^{\dagger}$  and $\omega_{]V_{H}[_{\mathfrak{P}_{H}}^{\dagger}}^{\bullet}:=\omega_{\mathfrak{P}_{H,K}}^{\bullet}|_{]V_{H}[_{\mathfrak{P}_{H}}^{\dagger}}$ are independent of the choice of exactification above (\cite{Gro05}, Lemma 1.2). Now, for $H_{1}\subset H_{2}$, the projection $p_{12}:]V_{H_{2}}[_{\mathfrak{P}_{H_{2}}}^{\dagger}\rightarrow ]V_{H_{1}}[_{\mathfrak{P}_{H_{1}}}^{\dagger}$ induces $p_{12}^{-1}\omega_{]V_{H_{1}}[_{\mathfrak{P}_{H_{1}}}^{\dagger}}^{\bullet}\rightarrow\omega_{]V_{H_{2}}[_{\mathfrak{P}_{H_{2}}}^{\dagger}}^{\bullet}$. This defines a complex of simplicial sheaves $\omega_{]V_{H_{\bullet}}[_{\mathfrak{P}_{H_{\bullet}}}^{\dagger}}^{\bullet}$ on a simplicial dagger space $]V_{H_{\bullet}}[_{\mathfrak{P}_{H_{\bullet}}}^{\dagger}$. The log-rigid cohomology of $Y$ is given by
\begin{equation*}
R\Gamma_{\logrig}(Y/\mathfrak{S}_{0}):=\mathbb{R}\Gamma(]V_{H_{\bullet}}[_{\mathfrak{P}_{H_{\bullet}}}^{\dagger},\omega_{]V_{H_{\bullet}}[_{\mathfrak{P}_{H_{\bullet}}}^{\dagger}}^{\bullet})
\end{equation*} 
Then $R\Gamma_{\logrig}(Y/\mathfrak{S}_{0})$ is independent on the choice of open cover $Y=\bigcup_{i\in I}V_{i}$ and the choice of embeddings $V_{i}\hookrightarrow\mathfrak{P}_{i}$ (\cite{Gro05},Lemma 1.4). 

\section{The overconvergent Hyodo-Kato complex}\label{complexes}

Now suppose that $k$ is a perfect field of characteristic $p>0$. We follow closely the approach of \cite{HK94}. Let $X$ be a regular flat $W$-scheme and write
\begin{center}
\begin{tikzpicture}[descr/.style={fill=white,inner sep=1.5pt}]
        \matrix (m) [
            matrix of math nodes,
            row sep=2.5em,
            column sep=2.5em,
            text height=1.5ex, text depth=0.25ex
        ]
        { Y & X & X_{K} \\
          \spec k & \spec W & \spec K\\
        };

        \path[overlay,right hook->, font=\scriptsize]
        (m-1-1) edge node [above]{$i$} (m-1-2)
        (m-2-1) edge (m-2-2);
        
        \path[overlay,left hook->, font=\scriptsize]
        (m-1-3) edge node [above]{$j$} (m-1-2)
        (m-2-3) edge (m-2-2);        
        
        \path[overlay,->, font=\scriptsize]
        (m-1-1) edge (m-2-1)
        (m-1-2) edge (m-2-2)
        (m-1-3) edge (m-2-3);
       \end{tikzpicture} 
\end{center} 
for the special and generic fibres of $X$. We suppose that $X$ has semistable reduction, that is to say we suppose that \'{e}tale locally on $X$, there is a smooth morphism $X\rightarrow\spec W[T_{1},\ldots,T_{n}]/(T_{1}\cdots T_{d}-p)$ for some $n\geq d$. In particular, $X_{K}$ is smooth and $Y$ is a reduced normal crossings divisor on $X$. If we endow $X$ with the log-structure induced by the special fibre and consider $Y$ with the pullback log structure, then $Y$ is a fine log-smooth $S_{0}$-log scheme. \'{E}tale locally on $Y$, the structure morphism factors as
\begin{equation*}
Y\xrightarrow{f}(\spec k[T_{1},\ldots,T_{n}]/(T_{1}\cdots T_{d}),\mathbb{N}^{d},e_{i}\mapsto T_{i})\xrightarrow{\delta}S_{0}
\end{equation*} 
where $f$ is exact and \'{e}tale and $\delta$ is induced by the diagonal. We say that $Y$ is semistable over $S_{0}$.

Since $k$ is a perfect field, we can find a dense open subscheme $u:U\hookrightarrow Y$ which is smooth over $k$. We may therefore consider the pushforward of the overconvergent de Rham-Witt complex $W^{\dagger}\Omega_{U/k}^{\bullet}$ of \cite{DLZ11}. Let
\begin{equation*}
d\log:i^{-1}j_{\ast}(\mathcal{O}_{X_{K}}^{\times})\rightarrow u_{\ast}W\Omega_{U/k}^{1}
\end{equation*} 
be the homomorphism considered in (\cite{HK94}, \S1). The image lies in $u_{\ast}W^{\dagger}\Omega_{U/k}^{1}$ (see \cite{Mat17}, 10.1). Then the Hyodo-Kato complex $W\omega_{Y/k}^{\bullet}$ is the $p$-adic completion of the $W(\mathcal{O}_{Y})$-subalgebra of $u_{\ast}W\Omega_{U/k}^{\bullet}$ generated by $dW(\mathcal{O}_{Y})$ and the image of $d\log$. Let $W^{\dagger}(\mathcal{O}_{Y})$ denote the Zariski sheaf of overconvergent Witt vectors (\cite{DLZ12}, Prop 3.2) on $Y$.

\begin{Def*}\sloppy
$W^{\dagger}\omega_{Y/k}^{\bullet}$ is the differential graded $W^{\dagger}(\mathcal{O}_{Y})$-algebra $W\omega_{Y/k}^{\bullet}\cap u_{\ast}W^{\dagger}\Omega_{U/k}^{\bullet}$.
\end{Def*}

Then $W^{\dagger}\omega_{Y/k}^{\bullet}$ is a subcomplex of the $W^{\dagger}(\mathcal{O}_{Y})$-algebra $u_{\ast}W^{\dagger}\Omega_{U/k}^{\bullet}$ and inherits the operators $F$ and $V$ satisfying the usual de Rham-Witt relations, as in \cite{HK94}.

Now let $u_{\ast}W^{\dagger}\Omega_{U/k}^{\bullet}[\theta]/(\theta^{2})$ be the complex given by adjoining an inderterminate $\theta$ in degree one, subject to $\theta a=(-1)^{q}a\theta$ for all $a\in u_{\ast}W^{\dagger}\Omega_{U/k}^{q}$ and $d\theta=0$. Let
\begin{equation*}
d\log:i^{-1}j_{\ast}(\mathcal{O}_{X_{K}}^{\times})\rightarrow u_{\ast}W^{\dagger}\Omega_{U/k}^{1}[\theta]/(\theta^{2})
\end{equation*} 
be the unique homomorphism which induces on $u^{-1}i^{-1}(\mathcal{O}_{X}^{\times})$ the composite map
\begin{equation*}
u^{-1}i^{-1}(\mathcal{O}_{X}^{\times})\rightarrow\mathcal{O}_{U}^{\times}\xrightarrow{d\log}W\Omega_{U/k}^{1}
\end{equation*} 
and induces on $K^{\times}$ the map $a\mapsto\text{ord}_{K}(a)\theta$ (again, see (\cite{HK94}, \S1). The image of $d\log$ lies by definition inside $W\tilde{\omega}_{Y/k}^{\bullet}$, defined in (\cite{HK94}, 1.4).
\begin{Def*}
$W^{\dagger}\tilde{\omega}_{Y/k}^{\bullet}$ is the $W^{\dagger}(\mathcal{O}_{Y})$-algebra $W\tilde{\omega}_{Y/k}^{\bullet}\cap u_{\ast}W^{\dagger}\Omega_{U/k}^{\bullet}[\theta]/(\theta^{2})$.
\end{Def*}
Then we have a short exact sequence of complexes, induced by (\cite{HK94}, Prop 1.5)
\begin{align}\label{ses DRW}
0\rightarrow W^{\dagger}\omega_{Y/k}^{\bullet}[-
& 1]\rightarrow W^{\dagger}\tilde{\omega}_{Y/k}^{\bullet}\rightarrow W^{\dagger}\omega_{Y/k}^{\bullet}\rightarrow 0 \\
& a\mapsto a\wedge\theta, \ \ \ \theta\mapsto 0 \nonumber
\end{align}

\subsection{An equivalent approach}\label{Matsuue}

In this section we shall outline another definition of the overconvergent Hyodo-Kato complex, this time in the style of \cite{Mat17}, and we will show that the two definitions are the same. This will become particularly useful in \S\ref{projective section}.

\sloppy
In (\cite{Mat17}, \S3.4), Matsuue defines the log de Rham-Witt complex $W\Lambda_{(S,Q)/(R,P)}^{\bullet}$ for any morphism of pre-log rings $(R,P)\rightarrow(S,Q)$, where $R$ is a $\mathbb{Z}_{(p)}$-algebra, as the initial object in the category of log F-V-procomplexes. The construction is a logarithmic generalisation of the construction given in (\cite{LZ04}, \S1.3).

Fix integers $n\geq d$ and let $(B:=k[T_{1},\ldots,T_{n}],\mathbb{N}^{d},e_{i}\mapsto T_{i})$ be considered as a pre-log ring over $(k,\{\ast\})$, where $\{\ast\}$ denotes the trivial monoid. Then one has, in particular, the log de Rham-Witt complex $W\Lambda_{(B,\mathbb{N}^{d})/(k,\{\ast\})}^{\bullet}$ as a special case. Any element of $W\Lambda_{(B,\mathbb{N}^{d})/(k,\{\ast\})}^{\bullet}$ can be written as a convergent sum of basic log Witt differentials (\cite{Mat17}, Prop. 4.3). Matsuue then defines a subcomplex $W^{\dagger}\Lambda_{(B,\mathbb{N}^{d})/(k,\{\ast\})}^{\bullet}$ as those elements of $W\Lambda_{(B,\mathbb{N}^{d})/(k,\{\ast\})}^{\bullet}$ which are overconvergent (\cite{Mat17}, \S10.1).

Now consider a pre-log ring $(A,M,\alpha)$ over $(k,\{\ast\})$ such that $A$ is a finitely generated $k$-algebra. Then we may choose a surjective morphism of pre-log rings over $(k,\mathbb{\{\ast\}})$:-
\begin{center}
\begin{tikzpicture}[descr/.style={fill=white,inner sep=1.5pt}]
        \matrix (m) [
            matrix of math nodes,
            row sep=2.5em,
            column sep=2.5em,
            text height=1.5ex, text depth=0.25ex
        ]
        { \mathbb{N}^{d} & B=k[T_{1},\ldots,T_{n}]  \\
          M & A \\
        };

        \path[overlay,->, font=\scriptsize]
        (m-1-1) edge (m-1-2)
        (m-2-1) edge node[above]{$\alpha$} (m-2-2);       
           
      \path[overlay, ->>, font=\scriptsize]
      (m-1-1) edge (m-2-1)
      (m-1-2) edge (m-2-2);
       \end{tikzpicture} 
\end{center} 
where the top morphism is $e_{i}\mapsto T_{i}$. This morphism of pre-log rings induces a morphism of log de Rham-Witt complexes 
\begin{equation*}
\lambda:W\Lambda_{(B,\mathbb{N}^{d})/{(k,\{\ast\})}}^{\bullet}\rightarrow W\Lambda_{(A,M)/{(k,\{\ast\})}}^{\bullet}
\end{equation*}
Matsuue defines $W^{\dagger}\Lambda_{(A,M)/{(k,\{\ast\})}}^{\bullet}:=\lambda\left(W^{\dagger}\Lambda_{(B,\mathbb{N}^{d})/{(k,\{\ast\})}}^{\bullet}\right)$. Notice that one could have taken any log polynomial algebra over $k$ which surjects onto $(A,M)$, but Matsuue shows that $W^{\dagger}\Lambda_{(A,M)/{(k,\{\ast\})}}^{\bullet}$ is independent of this choice (see (\cite{Mat17}, 10.2) and the subsequent discussion). By construction, $W^{\dagger}\Lambda_{(A,\{\ast\})/(k,\{\ast\})}^{\bullet}$ is the overconvergent de Rham-Witt complex $W^{\dagger}\Omega_{A/k}^{\bullet}$ of \cite{DLZ11}. In (\cite{DLZ11}, Cor 1.7) it is shown that this construction glues to give a complex of Zariski sheaves $W^{\dagger}\Omega_{X/k}^{\bullet}$ on any variety $X$. In (\cite{Mat17}, \S10.3) this is generalised to show that $W^{\dagger}\Lambda^{\bullet}_{(A,M)/(k,\{\ast\})}$ glues to give a complex of Zariski sheaves $W^{\dagger}\Lambda_{(X,M)/(k,\{\ast\})}^{\bullet}$ where $(X,M)$ denotes the log scheme associated to the complement of a strict normal crossing divisor. We give a similar argument in the semistable case.

Let $\spec A$ be a semistable affine $k$-scheme and let $(A,M,\alpha)$ be the associated pre-log ring. Define $P_{j}W\Lambda_{(A,M)/(k,\{\ast\})}^{r}$ to be the image of the map
\begin{equation*}
W\Lambda_{(A,M)/(k,\{\ast\})}^{j}\otimes W\Omega_{A/k}^{r-j}\rightarrow W\Lambda_{(A,M)/(k,\{\ast\})}^{r}
\end{equation*}  
This gives a filtration $P_{\bullet}W\Lambda_{(A,M)/(k,\{\ast\})}^{\bullet}$ of the complex $W\Lambda_{(A,M)/(k,\{\ast\})}^{\bullet}$. Let $\{\spec A_{i}\}_{i\in I}$ be the irreducible components of $\spec A$. For subsets $J\subset I$ let $\bigcap_{i\in J}\spec A_{i}=\spec A_{J}$. Then the Poincar\'{e} residue maps give 
\begin{equation*}
\mathrm{Gr}_{j}W\Lambda_{(A,M)/(k,\{\ast\})}^{\bullet}\simeq\bigoplus_{\stackrel{J\subset I}{|J|=j}}W\Omega_{A_{J}/k}^{\bullet}[-j]
\end{equation*} 
Similarly, define $P_{j}W^{\dagger}\Lambda_{(A,M)/(k,\{\ast\})}^{r}$ to be the image of the map
\begin{equation*}
W^{\dagger}\Lambda_{(A,M)/(k,\{\ast\})}^{j}\otimes W^{\dagger}\Omega_{A/k}^{r-j}\rightarrow W^{\dagger}\Lambda_{(A,M)/(k,\{\ast\})}^{r}
\end{equation*}  
to get a filtration $P_{\bullet}W^{\dagger}\Lambda_{(A,M)/(k,\{\ast\})}^{\bullet}$ of the complex $W^{\dagger}\Lambda_{(A,M)/(k,\{\ast\})}^{\bullet}$. The argument in (\cite{Mat17}, Lemma 10.9) shows that the above residue isomorphism induces
\begin{equation*}
\mathrm{Gr}_{j}W^{\dagger}\Lambda_{(A,M)/(k,\{\ast\})}^{\bullet}\simeq\bigoplus_{\stackrel{J\subset I}{|J|=j}}W^{\dagger}\Omega_{A_{J}/k}^{\bullet}[-j]
\end{equation*} 
\begin{Proposition}
Let $\spec A$ be a semistable affine $k$-scheme and let $(A,M,\alpha)$ be the associated pre-log ring. Then the presheaf canonically determined by 
\begin{equation*}
D(f)\mapsto W^{\dagger}\Lambda_{(A_{f},M)/(k,\{\ast\})}^{r}
\end{equation*}
on the basis of distinguished opens is a Zariski sheaf on $\spec A$. (Given $f\in A$, we consider the localisation $A_{f}$ as a pre-log ring via the composition $M\xrightarrow{\alpha}A\rightarrow A_{f}$.) 
\end{Proposition}\label{Cech}
\begin{proof}
The proof is similar to (\cite{DLZ11}, Prop 1.2) and (\cite{Mat17}, Prop 10.12). Let $f_{1},\ldots,f_{l}\in A$ be a generating set for $A$. Write $A_{i_{1}\cdots i_{s}}$ for $A_{f_{i_{1}}\cdots f_{i_{s}}}$. Consider the \v{C}ech complex $C^{\bullet}$ given by
\begin{equation*}
C^{s}=\bigoplus_{1\leq i_{1}<\cdots <i_{s}\leq l}W^{\dagger}\Lambda_{(A_{i_{1}\cdots i_{s}},M)/(k,\{\ast\})}^{r}
\end{equation*}
(so $C^{0}=W^{\dagger}\Lambda_{(A,M)/(k,\{\ast\})}^{r}$). Then it suffices to show that $C^{\bullet}$ is exact. Define
\begin{equation*}
P_{j}C^{s}=\bigoplus_{1\leq i_{1}<\cdots <i_{s}\leq l}P_{j}W^{\dagger}\Lambda_{(A_{i_{1}\cdots i_{s}},M)/(k,\{\ast\})}^{r}
\end{equation*}
This gives a filtration $P_{\bullet}C^{\bullet}$ of the complex $C^{\bullet}$. The graded piece $\mathrm{Gr}_{j}C^{\bullet}$ is
\begin{align*}
\mathrm{Gr}_{j}C^{s}
& \cong\bigoplus_{1\leq i_{1}<\cdots <i_{s}\leq l}\bigoplus_{\stackrel{J\subset I}{|J|=j}}W^{\dagger}\Omega_{(A_{i_{1}\cdots i_{s})_{J}}/k}^{r-j} \\
& \cong\bigoplus_{1\leq i_{1}<\cdots <i_{s}\leq l}\bigoplus_{\stackrel{J\subset I}{|J|=j}}W^{\dagger}\Omega_{(A_{J})_{i_{1}\cdots i_{s}}/k}^{r-j}
\end{align*}
in degree $s$. Here the second direct sum in the first line runs over all $j$-fold intersections of irreducible components of $\spec A_{i_{1}\cdots i_{s}}$. Therefore 
\begin{equation*}
\mathrm{Gr}_{j}C^{\bullet}\cong\bigoplus_{\stackrel{J\subset I}{|J|=j}}\tilde{C}^{\bullet}
\end{equation*}
where 
\begin{equation*}
\tilde{C}^{s}=\bigoplus_{1\leq i_{1}<\cdots <i_{s}\leq l}W^{\dagger}\Omega^{r-j}_{(A_{J})_{i_{1}\cdots i_{s}}/k}
\end{equation*}
is the \v{C}ech complex for $W^{\dagger}\Omega^{r-j}_{A_{J}/k}$. This is exact by (\cite{DLZ11}, Prop 1.6). Induction using the short exact sequence of complexes
\begin{equation*}
0\rightarrow P_{j-1}C^{\bullet}\rightarrow P_{j}C^{\bullet}\rightarrow\mathrm{Gr}_{j}C^{\bullet}\rightarrow 0
\end{equation*}
shows that $P_{j}C^{\bullet}$ is exact for all $j$, and hence $C^{\bullet}$ is exact.
\end{proof}

One may therefore glue to define a complex of Zariski sheaves $W^{\dagger}\Lambda_{Y/{(k,\{\ast\})}}^{\bullet}$ for semistable schemes $Y$ over $S_{0}$ (recall from \S\ref{log-rigid section} that $S_{0}$ is the standard log point $(\spec k, \mathbb{N},1\mapsto 0)$). Finally, we define the overconvergent Hyodo-Kato complex $W^{\dagger}\Lambda_{Y/S_{0}}^{\bullet}$ to be the image of $W^{\dagger}\Lambda_{Y/{(k,\{\ast\})}}^{\bullet}$ under the projection
\begin{equation*}
W\Lambda_{Y/{(k,\{\ast\})}}^{\bullet}\rightarrow W\Lambda_{Y/(k,\mathbb{N})}^{\bullet}=W\Lambda_{Y/S_{0}}^{\bullet}
\end{equation*}
of log de Rham-Witt complexes. This is again a complex of Zariski sheaves.
\begin{Proposition}\label{old-modern}
The overconvergent Hyodo-Kato complex $W^{\dagger}\omega_{Y/k}^{\bullet}$ of the previous section is the same as $W^{\dagger}\Lambda_{Y/S_{0}}^{\bullet}$.
\end{Proposition}
\begin{proof}
We first show that Matsuue's log de Rham-Witt complex $W\Lambda_{Y/S_{0}}^{\bullet}$ agrees with the Hyodo-Kato complex $W\omega_{Y/k}^{\bullet}$ of \cite{HK94}. This must be well-known to the experts, but the authors do not know of a proof recorded in the literature; it seems important to reconcile the two approaches, so we give a proof here. Since both complexes are complexes of Zariski sheaves it suffices to construct a canonical isomomorphism, functorial in $Y$, if $Y$ is affine.

Given a log $F$-$V$ procomplex $\{E_{m}^{\bullet}\}_{m\in\mathbb{N}}$, define differential graded ideals
\begin{equation*}
\fil^{s}E_{m}^{i}:=V^{s}E_{m-s}^{i}+dV^{s}E_{m-s}^{i-1}\subset E_{m}^{i}
\end{equation*} 
Then this gives a filtration of log $F$-$V$-procomplexes which is compatible with $F$, $V$, $d$ and the projections (\cite{Mat17}, \S3.5).

Now, $W_{\bullet}\Lambda_{Y/S_{0}}^{\bullet}=\left\{W_{m}\Lambda_{Y/S_{0}}^{\bullet}\right\}_{m\in \mathbb{N}}$ and $W_{\bullet}\omega_{Y/k}^{\bullet}=\left\{W_{m}\omega_{Y/k}^{\bullet}\right\}_{m\in \mathbb{N}}$ are log F-V-procomplexes, so we have a map of log F-V-procomplexes, evidently functorial in $Y$,  
\begin{equation*}
W_{\bullet}\Lambda_{Y/S_{0}}^{\bullet}\rightarrow W_{\bullet}\omega_{Y/k}^{\bullet}
\end{equation*}
by the universal property of $W\Lambda_{Y/S_{0}}^{\bullet}$. This map induces diagrams
\begin{equation*}
\begin{tikzpicture}[descr/.style={fill=white,inner sep=1.5pt}]
        \matrix (m) [
            matrix of math nodes,
            row sep=2.5em,
            column sep=2.5em,
            text height=1.5ex, text depth=0.25ex
        ]
        {
        0 & \fil^{m}W_{m+1}\Lambda_{Y/S_{0}}^{\bullet} & W_{m+1}\Lambda_{Y/S_{0}}^{\bullet} & W_{m}\Lambda_{Y/S_{0}}^{\bullet} & 0 \\
        0 & \fil^{m}W_{m+1}\omega_{Y/k}^{\bullet} & W_{m+1}\omega_{Y/k}^{\bullet} & W_{m}\omega_{Y/k}^{\bullet} & 0  \\
        };

        \path[overlay,->, font=\scriptsize]
       (m-1-1) edge (m-1-2)
       (m-1-2) edge (m-1-3)
       (m-1-3) edge (m-1-4)
       (m-1-4) edge (m-1-5)
       (m-2-1) edge (m-2-2)
       (m-2-2) edge (m-2-3)
       (m-2-3) edge (m-2-4)
       (m-2-4) edge (m-2-5)
       (m-1-2) edge (m-2-2)
       (m-1-3) edge (m-2-3)
       (m-1-4) edge (m-2-4);

       \end{tikzpicture} 
\end{equation*} 
of short exact sequences (see (\cite{Mat17}, Prop. 3.6) for the top row, and (\cite{HK94}, Theorem 4.4) for the bottom row) for each $m\in\mathbb{N}$. Now one notices that 
\begin{equation*}
W_{1}\Lambda_{Y/S_{0}}^{\bullet}=W_{1}\omega_{Y/k}^{\bullet}=\omega_{Y/k}^{\bullet}
\end{equation*}
is the usual logarithmic de Rham complex, by definition. This gives 
\begin{equation*}
\fil^{m}W_{m+1}\Lambda_{Y/S_{0}}^{\bullet}=\fil^{m}W_{m+1}\omega_{Y/k}^{\bullet}
\end{equation*}
and then the diagrams give $W_{m}\Lambda_{Y/S_{0}}^{\bullet}=W_{m}\omega_{Y/k}^{\bullet}$ for all $m\in\mathbb{N}$.

Since $W^{\dagger}\omega_{Y/k}^{\bullet}$ and $W^{\dagger}\Lambda_{Y/S_{0}}^{\bullet}$ are subcomplexes of Zariski sheaves of the completed versions, it suffices to show the claim for $Y$ affine. By a result of Kedlaya (\cite{Ked05}, Theorem 2), $Y$ can be covered by finitely many affines $\spec B_{i}$ such that $B_{i}$ is finite \'{e}tale and free over $A_{i}=k[T_{1},\ldots, T_{d}]/(T_{1}\cdots T_{r})$ for some $r$. Using again a sheaf argument, it suffices to prove the claim for $B$ finite \'{e}tale and free over $A=k[T_{1},\ldots, T_{d}]/(T_{1}\cdots T_{r})$. By \'{e}tale base change (\cite{Mat17}, Prop 3.7) and the fact that $W(B)$ is again finite \'{e}tale and free over $W(A)$, we have $W\Lambda_{B/S_{0}}^{\bullet}=W(B)\otimes_{W(A)}W\Lambda_{A/S_{0}}^{\bullet}$ and likewise $W\omega_{B/k}^{\bullet}=W(B)\otimes_{W(A)}W\omega_{A/k}^{\bullet}$. Note that the isomorphism $W\Lambda_{A/S_{0}}^{\bullet}\cong W\omega_{A/k}^{\bullet}$ is explicitly given by formula \eqref{unique expression} in \S\ref{projective section}.

Now, (\cite{DLZ12}, Cor 2.46) implies that $W^{\dagger}(B)$ is finite \'{e}tale and free as a $W^{\dagger}(A)$-module. The proofs of (\cite{DLZ11}, Prop 1.9) or, alternatively, (\cite{DLZ11}, Prop 3.19) transfer to the Hyodo-Kato complexes and provide \'{e}tale base change in the overconvergent settting:
\begin{align*}
& W^{\dagger}\Lambda_{B/S_{0}}^{\ell}\cong W^{\dagger}\Lambda_{A/S_{0}}^{\ell}\otimes_{W^{\dagger}(A)}W^{\dagger}(B) \\
& W^{\dagger}\omega_{B/k}^{\ell}\cong W^{\dagger}\omega_{A/k}^{\ell}\otimes_{W^{\dagger}(A)}W^{\dagger}(B)
\end{align*}
This finishes the proof, since evidently by definition $W^{\dagger}\Lambda_{A/S_{0}}^{\ell}\cong W^{\dagger}\omega_{A/k}^{\ell}$.
\end{proof}

\section{Comparison with log-Monsky-Washniter cohomology}\label{logmw section}

Let $Y=\spec A$ be a semistable affine scheme over $S_{0}$. Let $Y\hookrightarrow Z=\spec B$ be a closed embedding into a smooth affine $k$-scheme such that $Y$ is a normal crossings divisor on $Z$, in other words $A=B/(f_{1}\cdots f_{r})$ and each $B/(f_{i})$ is smooth. Let $\tilde{B}$ be a smooth $W$-algebra lifting $B$ (there always exists such a $\tilde{B}$ by \cite{Elk73}) and set $\tilde{A}:=\tilde{B}/(\tilde{f}_{1}\cdots\tilde{f}_{r})$ for some liftings $\tilde{f}_{i}\in\tilde{B}$ of the $f_{i}$, such that $\tilde{Y}:=\spec\tilde{A}$ is a normal crossings divisor in $\tilde{Z}=\spec\tilde{B}$. That is, we have a diagram 
\begin{center}
\begin{tikzpicture}[descr/.style={fill=white,inner sep=1.5pt}]
        \matrix (m) [
            matrix of math nodes,
            row sep=2.5em,
            column sep=2.5em,
            text height=1.5ex, text depth=0.25ex
        ]
        {\tilde{Y}=\spec\tilde{A} & \tilde{Z}=\spec\tilde{B}  \\
          Y=\spec A & Z=\spec B \\
        };

        \path[overlay,->, font=\scriptsize]
        (m-1-1) edge (m-2-1)
        (m-1-2) edge (m-2-2);

        \path[overlay,right hook->, font=\scriptsize]
        (m-1-1) edge (m-1-2)
        (m-2-1) edge (m-2-2);

       \end{tikzpicture} 
\end{center} 

We define the complexes $W^{\dagger}\omega_{Y/k}^{\bullet}$ and $W^{\dagger}\tilde{\omega}_{Y/k}^{\bullet}$ as in \S\ref{complexes}. Indeed, $X=\spec\tilde{B}/(\tilde{f}_{1}\cdots\tilde{f}_{r}-p)$ is a regular scheme whose special fibre is $Y$, so the definition applies. 

Now let $\Omega_{\tilde{Z}/W}^{\bullet}(\log\tilde{Y})$ denote the logarithmic de Rham complex of $\tilde{Z}$ with respect to the normal crossings divisor $\tilde{Y}$.  We set 
\begin{equation*}
\tilde{\omega}_{\tilde{Y}}^{\bullet}:=\mathcal{O}_{\tilde{Y}}\otimes_{\mathcal{O}_{\tilde{Z}}}\Omega_{\tilde{Z}/W}^{\bullet}(\log\tilde{Y})
\end{equation*}
and write
\begin{equation*}
\tilde{\omega}_{\tilde{Y}^{\dagger}}^{\bullet}:=\mathcal{O}_{\tilde{Y}^{\dagger}}\otimes_{\mathcal{O}_{\tilde{Z}^{\dagger}}}\Omega_{\tilde{Z}^{\dagger}/W}^{\bullet}(\log\tilde{Y}^{\dagger})
\end{equation*}
for the weak completion. 
\begin{Def*}
The logarithmic Monsky-Washnitzer complex of $Y$ is defined to be
\begin{equation*}
\omega_{\tilde{Y}^{\dagger}}^{\bullet}:=\tilde{\omega}_{\tilde{Y}^{\dagger}}^{\bullet}/(\tilde{\omega}_{\tilde{Y}^{\dagger}}^{\bullet-1}\wedge\theta)
\end{equation*}
where $\theta:=d\log\tilde{f}_{1}+\cdots+d\log\tilde{f}_{r}$. We define
\begin{equation*}
H_{\emph{\logmw}}^{\ast}(Y/K):=\mathbb{H}^{\ast}(Y,\omega_{\tilde{Y}^{\dagger}}^{\bullet}\otimes\mathbb{Q})
\end{equation*}
\end{Def*}
This is the logarithmic Monsky-Washnitzer cohomology, as discussed in (\cite{Gro05}, \S5). It is clear that 
\begin{equation*}
H_{\logmw}^{\ast}(Y/K)\cong H_{\logrig}^{\ast}(Y/\mathfrak{S}_{0})
\end{equation*}
and that there is a short exact sequence of complexes
\begin{align}\label{ses MW}
0\rightarrow \omega_{\tilde{Y}^{\dagger}}^{\bullet}[
& -1]\rightarrow\tilde{\omega}_{\tilde{Y}^{\dagger}}^{\bullet}\rightarrow\omega_{\tilde{Y}^{\dagger}}^{\bullet}\rightarrow 0 \\
& a\mapsto a\wedge\theta,\ \theta\mapsto 0 \nonumber
\end{align}

In this section we shall construct a morphism of short exact sequences from (\ref{ses MW}) to (\ref{ses DRW}). We will then prove that the subsequent vertical arrows become quasi-isomorphisms after tensoring with $\mathbb{Q}$.

Let $\tilde{A}^{\dagger}$ and $\tilde{B}^{\dagger}$ denote the weak completion of $\tilde{A}$ and $\tilde{B}$, respectively. Then we have an induced diagram
\begin{center}
\begin{tikzpicture}[descr/.style={fill=white,inner sep=1.5pt}]
        \matrix (m) [
            matrix of math nodes,
            row sep=2.5em,
            column sep=2.5em,
            text height=1.5ex, text depth=0.25ex
        ]
        { \tilde{B}^{\dagger} & \tilde{A}^{\dagger} \\
            W(B) & W(A)  \\
        };

        \path[overlay,->, font=\scriptsize]
        (m-1-1) edge node [left] {$t_{F}$} (m-2-1)
        (m-1-2) edge node [right] {$t_{F}$} (m-2-2);

        \path[overlay,->>, font=\scriptsize]
        (m-1-1) edge (m-1-2)
        (m-2-1) edge (m-2-2);

       \end{tikzpicture} 
\end{center} 
where the vertical arrows are the Lazard morphisms (\cite{Ill79} 0 1.3.6). Since $B$ is a smooth finitely generated $k$-algebra, $t_{F}:\tilde{B}^{\dagger}\rightarrow W(B)$ has image contained in the overconvergent Witt vectors $W^{\dagger}(B)$ (\cite{DLZ11}, Prop. 3.2). By functoriality of the Lazard morphisms and since $W(B)\twoheadrightarrow W(A)$ sends $W^{\dagger}(B)$ to $W^{\dagger}(A)$, we deduce that we in fact have a diagram
\begin{equation*}
\begin{tikzpicture}[descr/.style={fill=white,inner sep=1.5pt}]
        \matrix (m) [
            matrix of math nodes,
            row sep=2.5em,
            column sep=2.5em,
            text height=1.5ex, text depth=0.25ex
        ]
        { \tilde{B}^{\dagger} & \tilde{A}^{\dagger} \\
            W^{\dagger}(B) & W^{\dagger}(A)  \\
        };

        \path[overlay,->, font=\scriptsize]
        (m-1-1) edge node [left] {$t_{F}$} (m-2-1)
        (m-1-2) edge node [right] {$t_{F}$} (m-2-2);

        \path[overlay,->>, font=\scriptsize]
        (m-1-1) edge (m-1-2)
        (m-2-1) edge (m-2-2);

       \end{tikzpicture} 
\end{equation*} 

Let $u$ and $\tilde{u}$ denote the respective open immersions $U:=Z\backslash Y\hookrightarrow Z$ and $\tilde{U}:=\tilde{Z}\backslash \tilde{Y}\hookrightarrow\tilde{Z}$. Then the map
\begin{align*}
& \tilde{u}_{\ast}\Omega_{\tilde{U}/W}^{\bullet}\rightarrow u_{\ast}W\Omega_{U/k}^{\bullet} \\
& \ \ \ d\log\tilde{f}_{i}\mapsto d\log[f_{i}]
\end{align*}
sends logarithmic differentials along $\tilde{Y}$ to logarithmic differentials along $Y$. (Here $[f_{i}]\in W(\tilde{B})$ denotes the Teichm\"{u}ller lift). In particular, it induces a map
\begin{equation*}
\Omega_{\tilde{Z}/W}^{\bullet}(\log\tilde{Y})\rightarrow W\Omega_{Z/k}^{\bullet}(\log Y)
\end{equation*}
and by the above discussion, this induces a map 
\begin{equation*}
\Omega_{\tilde{Z}^{\dagger}/W}^{\bullet}(\log\tilde{Y}^{\dagger})\rightarrow W^{\dagger}\Omega_{Z/k}^{\bullet}(\log Y)
\end{equation*}
These maps were considered in (\cite{Mat17}, \S10), and become quasi-isomorphisms after tensoring with $\mathbb{Q}$, by (\cite{Mat17}, Lemma 10.9). In any case, this gives a map
\begin{equation}\label{map}
\tilde{\omega}_{\tilde{Y}^{\dagger}}^{\bullet}\rightarrow W^{\dagger}(A)\otimes_{W^{\dagger}(B)}W^{\dagger}\Omega_{Z/k}^{\bullet}(\log Y)
\end{equation}
Notice that the logarithmic differentials $d\log[f_{i}]$ along $Y$ coincide with the logarithmic differentials as defined by Hyodo-Kato as the image of $d\log$ using the regular $W$-scheme $X=\spec\tilde{B}/(\tilde{f}_{1}\cdots\tilde{f}_{r}-p)$. Indeed $\tilde{f}_{i}$ is mapped to $d\log [f_{i}]$ and $p$ is mapped to $\theta=d\log[f_{1}]+\cdots+d\log[f_{r}]$. 

Let $W_{n}\tilde{\omega}_{Y/k}^{i}$ be the sheaves introduced in (\cite{Hyo91}, \S1.6). These are the same as $W_{n}\Lambda_{Y/(k,\{\ast\})}^{i}$. One way to see this is by mimicking the proof of Proposition \ref{old-modern}. Recall that one can express the sheaves $W_{n}\tilde{\omega}_{Y/k}^{i}$ as quotients of the $W_{n}\Omega_{Z/k}^{i}(\log Y)$. Indeed, we may assume that $\tilde{Z}$ is an admissible lifting of $Y$ (see (\cite{Mok93}, \S2.4) for the definition). Set $\tilde{Z}_{n}:=\tilde{Z}\times_{W}W_{n}$ and $\tilde{Y}_{n}:=\tilde{Y}\times_{W}W_{n}$, so that $Z=\tilde{Z}_{1}$ and $Y=\tilde{Y}_{1}$. Then 
\begin{equation}\label{Hyodo-style quotient}
W_{n}\tilde{\omega}_{Y/k}^{i}=W_{n}\Omega_{Z/k}^{i}(\log Y)/W_{n}\Omega_{Z/k}^{i}(-\log Y)
\end{equation}
where we have identified $W_{n}\Omega_{Z/k}^{i}(\log Y)=\mathcal{H}^{i}(\Omega_{\tilde{Z}_{n}/W_{n}}^{\bullet}(\log\tilde{Y}_{n}))$ and $W_{n}\Omega_{Z/k}^{i}(-\log Y)=\mathcal{H}^{i}(\mathcal{J}_{n}\otimes_{\mathcal{O}_{\tilde{Z}_{n}}}\Omega_{\tilde{Z}_{n}/W_{n}}^{\bullet}(\log\tilde{Y}_{n}))$, where $\mathcal{J}_{n}=\ker(\mathcal{O}_{\tilde{Z}_{n}}\rightarrow\mathcal{O}_{\tilde{Y}_{n}})$. The fact that the right-hand side of (\ref{Hyodo-style quotient}) gives the same sheaf $W_{n}\tilde{\omega}_{Y/k}^{i}$ as defined in \cite{HK94} is discussed in (\cite{Mok93}, \S2.4). Passing to the  projective limit gives $W\tilde{\omega}_{Y/k}^{i}$ as a quotient of $W\Omega_{Z/k}^{i}(\log Y)$. We therefore deduce a map of complexes of sheaves on $Y$
\begin{equation*}
W(A)\otimes_{W(B)}W\Omega_{Z/k}^{\bullet}(\log Y)\rightarrow W\tilde{\omega}_{Y/k}^{\bullet}
\end{equation*}
Let $\Sigma$ be the singular locus of $Y$, $U=Y\setminus\Sigma$, $u:U\hookrightarrow Y$ and $i:Y\rightarrow Z$. Then we have canonical maps
\begin{equation*}
W^{\dagger}\Omega_{Z/k}^{\bullet}\rightarrow i_{\ast}W^{\dagger}\Omega_{Y/k}^{\bullet}\rightarrow i_{\ast}u_{\ast}W^{\dagger}\Omega_{U/k}^{\bullet}
\end{equation*}
which extend to a map
\begin{equation*}
W^{\dagger}\Omega_{Z/k}^{\bullet}(\log Y)\rightarrow i_{\ast}u_{\ast}W^{\dagger}\Omega_{U/k}^{\bullet}[\theta]/(\theta^{2})
\end{equation*}
since the $d\log[f_{i}]$ are overconvergent (\cite{Mat17}, 10.1). The image of this map lies -- by the above interpretation of $W\tilde{\omega}_{Y/k}^{\bullet}$ -- in
\begin{equation*}
i_{\ast}W^{\dagger}\tilde{\omega}_{Y/k}^{\bullet}=i_{\ast}W\tilde{\omega}_{Y/k}^{\bullet}\cap i_{\ast}u_{\ast}W^{\dagger}\Omega_{U/k}^{\bullet}[\theta]/(\theta^{2})
\end{equation*}
We get in each degree a $W^{\dagger}(B)$-module map
\begin{equation*}
W^{\dagger}\Omega_{Z/k}^{r}(\log Y)\rightarrow i_{\ast}W^{\dagger}\tilde{\omega}_{Y/k}^{r}
\end{equation*}
Hence we get a canonical map
\begin{equation*}
W^{\dagger}(A)\otimes_{W^{\dagger}(B)}W^{\dagger}\Omega_{Z/k}^{\bullet}(\log Y)\rightarrow W^{\dagger}\tilde{\omega}_{Y/k}^{\bullet}
\end{equation*}
Composing with (\ref{map}) defines a comparison morphism
\begin{equation}\label{comparison tilde}
\tilde{\omega}_{\tilde{Y}^{\dagger}}^{\bullet}\rightarrow W^{\dagger}\tilde{\omega}_{Y/k}^{\bullet}
\end{equation}
which sends $\theta=d\log\tilde{f}_{1}+\cdots+d\log\tilde{f}_{r}$ to $\theta=d\log[f_{1}]+\cdots +d\log[f_{r}]$. Since the ``divide by $\theta$'' projection $W^{\dagger}\tilde{\omega}_{Y/k}^{\bullet}\twoheadrightarrow W^{\dagger}\omega_{Y/k}^{\bullet}$ sends $\theta$ to $0$, we get an induced comparison morphism 
\begin{equation}\label{comparison}
\omega_{\tilde{Y}^{\dagger}}^{\bullet}\rightarrow W^{\dagger}\omega_{Y/k}^{\bullet}
\end{equation}
between the logarithmic Monsky-Washnitzer and overconvergent Hyodo-Kato complexes. Moreover, (\ref{comparison tilde}) and (\ref{comparison}) give a diagram of exact rows
\begin{equation*}
\begin{tikzpicture}[descr/.style={fill=white,inner sep=1.5pt}]
        \matrix (m) [
            matrix of math nodes,
            row sep=2.5em,
            column sep=2.5em,
            text height=1.5ex, text depth=0.25ex
        ]
        {0 & \omega_{\tilde{Y}^{\dagger}}^{\bullet}[-1] & \tilde{\omega}_{\tilde{Y}^{\dagger}}^{\bullet} &  \omega_{\tilde{Y}^{\dagger}}^{\bullet} & 0  \\
          0 & W^{\dagger}\omega_{Y/k}^{\bullet}[-1] & W^{\dagger}\tilde{\omega}_{Y/k}^{\bullet} & W^{\dagger}\omega_{Y/k}^{\bullet} & 0 \\
        };

        \path[overlay,->, font=\scriptsize]
       (m-1-1) edge (m-1-2)
       (m-1-2) edge (m-1-3)
       (m-1-3) edge (m-1-4)
       (m-1-4) edge (m-1-5)
       (m-2-1) edge (m-2-2)
       (m-2-2) edge (m-2-3)
       (m-2-3) edge (m-2-4)
       (m-2-4) edge (m-2-5)
       (m-1-2) edge (m-2-2)
       (m-1-3) edge (m-2-3)
       (m-1-4) edge (m-2-4);

       \end{tikzpicture} 
\end{equation*} 
We will use the weight filtration of Steenbrink to show that the vertical arrows (\ref{comparison tilde}) and (\ref{comparison}) become quasi-isomorphisms after tensoring with $\mathbb{Q}$. 

\begin{Theorem}\label{rational quasi-isomorphisms}
The comparison morphisms (\ref{comparison tilde}) and (\ref{comparison}) induce quasi-isomorphisms 
\begin{equation*}
\tilde{\omega}_{\tilde{Y}^{\dagger}}^{\bullet}\otimes\mathbb{Q}\xrightarrow{\sim} W^{\dagger}\tilde{\omega}_{Y/k}^{\bullet}\otimes\mathbb{Q}
\end{equation*}
and
\begin{equation*}
\omega_{\tilde{Y}^{\dagger}}^{\bullet}\otimes\mathbb{Q}\xrightarrow{\sim} W^{\dagger}\omega_{Y/k}^{\bullet}\otimes\mathbb{Q}
\end{equation*}
\end{Theorem}
\begin{proof}
Recall that the weight filtration $P_{\bullet}\tilde{\omega}_{\tilde{Y}^{\dagger}}^{\bullet}$ of $\tilde{\omega}_{\tilde{Y}^{\dagger}}^{\bullet}$ (see (\cite{Gro05}, \S5) for it in this context) is defined as
\begin{equation*}
P_{j}\tilde{\omega}_{\tilde{Y}^{\dagger}}^{i}:=\text{image}\left(\Omega_{\tilde{Z}^{\dagger}/W}^{j}(\log\tilde{Y})\otimes\Omega_{\tilde{Z}^{\dagger}/W}^{i-j}\rightarrow\Omega_{\tilde{Z}^{\dagger}/W}^{i}(\log\tilde{Y})\right)\otimes_{\mathcal{O}_{\tilde{Z}^{\dagger}}}\mathcal{O}_{\tilde{Y}^{\dagger}}
\end{equation*}
Via the Poincar\'{e} residue maps, the graded pieces of the filtration are identified as
\begin{equation*}
\text{Gr}_{j}(\tilde{\omega}_{\tilde{Y}^{\dagger}}^{\bullet}\otimes\mathbb{Q})\xrightarrow{\sim}\bigoplus_{Y_{I}\in\mathcal{M}_{j}}\Omega_{]Y_{I}[_{\tilde{Z}^{\dagger}}^{\dagger}}^{\bullet}[-j]
\end{equation*}
where $\mathcal{M}_{j}$ denotes the collection of all (smooth) intersections of $j$ different components of $Y$ which lift to a smooth intersection of $j$ different liftings in $\tilde{Y}$, and where $\Omega_{]Y_{I}[_{\tilde{Z}^{\dagger}}^{\dagger}}^{\bullet}$ denotes the usual de Rham complex on the smooth affinoid dagger space $]Y_{I}[_{\tilde{Z}^{\dagger}}^{\dagger}$. By (\cite{Gro05}, \S5.2), one has an isomorphism of exact sequences
\begin{equation*}
\begin{adjustbox}{width=12cm}
\begin{tikzpicture}[descr/.style={fill=white,inner sep=1.5pt}]
        \matrix (m) [
            matrix of math nodes,
            row sep=2.5em,
            column sep=2em,
            text height=1.5ex, text depth=0.25ex
        ]
        {0 & \text{Gr}_{0}(\tilde{\omega}_{\tilde{Y}^{\dagger}}^{\bullet}\otimes\mathbb{Q}) & \text{Gr}_{1}(\tilde{\omega}_{\tilde{Y}^{\dagger}}^{\bullet}\otimes\mathbb{Q})[1] & \text{Gr}_{2}(\tilde{\omega}_{\tilde{Y}^{\dagger}}^{\bullet}\otimes\mathbb{Q})[2] & \cdots  \\
          0 & \Omega_{\tilde{Y}^{\dagger}}^{\bullet}\otimes\mathbb{Q} & \bigoplus\limits_{Y_{I}\in\mathcal{M}_{1}}\Omega_{]Y_{I}[_{\tilde{Z}^{\dagger}}^{\dagger}}^{\bullet}\otimes\mathbb{Q} & \bigoplus\limits_{Y_{I}\in\mathcal{M}_{2}}\Omega_{]Y_{I}[_{\tilde{Z}^{\dagger}}^{\dagger}}^{\bullet}\otimes\mathbb{Q} & \cdots\\
        };

        \path[overlay,->, font=\scriptsize]
       (m-1-1) edge (m-1-2)
       (m-1-2) edge node [above]{$\wedge\theta$} (m-1-3)
       (m-1-3) edge node [above]{$\wedge\theta$} (m-1-4)
       (m-1-4) edge node [above]{$\wedge\theta$} (m-1-5)
       (m-2-1) edge (m-2-2)
       (m-2-2) edge (m-2-3)
       (m-2-3) edge (m-2-4)
       (m-2-4) edge (m-2-5)
       (m-1-2) edge node [right] {$\wr$} (m-2-2)
       (m-1-3) edge node [right] {$\wr$} (m-2-3)
       (m-1-4) edge node [right] {$\wr$} (m-2-4);

       \end{tikzpicture} 
       \end{adjustbox}
\end{equation*} 
(the bottom row is exact because the $Y_{I}$'s are normal crossings intersections).

Similarly, consider the weight filtration of Mokrane \cite{Mok93} on $W\tilde{\omega}_{Y/k}^{\bullet}$
\begin{equation*}
P_{j}W\tilde{\omega}_{Y/k}^{i}:=\text{image}\left(W\tilde{\omega}_{Y/k}^{j}\otimes W\Omega_{Y/k}^{i-j}\rightarrow W\tilde{\omega}_{Y/k}^{i}\right)
\end{equation*}
and set 
\begin{equation*}
P_{j}W^{\dagger}\tilde{\omega}_{Y/k}^{i}:=\text{image}\left(W^{\dagger}\tilde{\omega}_{Y/k}^{j}\otimes W^{\dagger}\Omega_{Y/k}^{i-j}\rightarrow W^{\dagger}\tilde{\omega}_{Y/k}^{i}\right)
\end{equation*}
By construction, the comparison morphism (\ref{comparison tilde}) induces maps $P_{j}\tilde{\omega}_{\tilde{Y}^{\dagger}}^{\bullet}\rightarrow P_{j}W^{\dagger}\tilde{\omega}_{Y/k}^{\bullet}$ for each $j$, and therefore repects the weight filtrations. Moreover, we have $P_{j}W^{\dagger}\tilde{\omega}_{Y/k}^{\bullet}=W^{\dagger}\tilde{\omega}_{Y/k}^{\bullet}\cap P_{j}W\tilde{\omega}_{Y/k}^{\bullet}$ for each $j$. By (\cite{Mok93}, 3.7), the graded pieces of the weight filtration are identified, via the Poincar\'{e} residue maps, as
\begin{equation*}
\text{Gr}_{j}W\tilde{\omega}_{Y/k}^{\bullet}\xrightarrow{\sim}\bigoplus_{Y_{I}\in\mathcal{M}_{j}}W\Omega_{Y_{I}/k}^{\bullet}[-j]
\end{equation*}
and therefore
\begin{equation*}
\text{Gr}_{j}W^{\dagger}\tilde{\omega}_{Y/k}^{\bullet}\xrightarrow{\sim}\bigoplus_{Y_{I}\in\mathcal{M}_{j}}W^{\dagger}\Omega_{Y_{I}/k}^{\bullet}[-j]
\end{equation*}

Since each $Y_{I}$ is smooth over $k$, we know by \cite{DLZ11} that $W^{\dagger}\Omega_{Y_{I}/k}^{\bullet}\otimes\mathbb{Q}$ is quasi-isomorphic to $\Omega_{]Y_{I}[_{\tilde{Z}^{\dagger}}}^{\bullet}$, and therefore conclude that the comparison morphism (\ref{comparison tilde}) is a quasi-isomorphism when tensored with  $\mathbb{Q}$.

To show that the second comparison morphism induces a quasi-isomorphism after tensoring with $\mathbb{Q}$, define a double complex 
\begin{equation*}
\mathcal{A}_{\mathbb{Q}}^{\dagger i,j}:=\frac{\tilde{\omega}^{i+j+1}_{\tilde{Y}^{\dagger}}\otimes\mathbb{Q}}{P_{j}\tilde{\omega}^{i+j+1}_{\tilde{Y}^{\dagger}}\otimes\mathbb{Q}}
\end{equation*}
with differential $\mathcal{A}_{\mathbb{Q}}^{\dagger i,j}\rightarrow\mathcal{A}_{\mathbb{Q}}^{\dagger i+1,j}$ induced by $(-1)^{j}d$,  and the other differential $\mathcal{A}_{\mathbb{Q}}^{\dagger i,j}\rightarrow\mathcal{A}_{\mathbb{Q}}^{\dagger i,j+1}$ induced by $\omega\mapsto\omega\wedge\theta$. Let $\mathcal{A}_{\mathbb{Q}}^{\dagger\bullet}$ be the total complex of $\mathcal{A}_{\mathbb{Q}}^{\dagger\bullet,\bullet}$. Entirely similarly, define another double complex $\mathcal{B}_{\mathbb{Q}}^{\dagger\bullet,\bullet}$ by
\begin{equation*}
\mathcal{B}_{\mathbb{Q}}^{\dagger i,j}:=\frac{W^{\dagger}\tilde{\omega}_{Y/k}^{i+j+1}\otimes\mathbb{Q}}{P_{j}W^{\dagger}\tilde{\omega}_{Y/k}^{i+j+1}\otimes\mathbb{Q}}
\end{equation*}
with the differential $\mathcal{B}_{\mathbb{Q}}^{\dagger i,j}\rightarrow\mathcal{B}_{\mathbb{Q}}^{\dagger i+1,j}$ induced by $(-1)^{j}d$ and the  differential $\mathcal{B}_{\mathbb{Q}}^{\dagger i,j}\rightarrow\mathcal{B}_{\mathbb{Q}}^{\dagger i,j+1}$ induced by $\omega\mapsto\omega\wedge\theta$, and let $\mathcal{B}_{\mathbb{Q}}^{\dagger\bullet}$ be the total complex of $\mathcal{B_{\mathbb{Q}}}^{\dagger\bullet,\bullet}$. Then $\mathcal{A}_{\mathbb{Q}}^{\dagger\bullet}$ is quasi-isomorphic to $\mathcal{B}_{\mathbb{Q}}^{\dagger\bullet}$, because the graded quotients $\Omega_{]Y_{I}[_{\tilde{Z}}^{\dagger}}^{\bullet}$ and $W^{\dagger}\Omega_{Y_{I}/k}^{\bullet}\otimes\mathbb{Q}$ are quasi-isomorphic by the comparison theorem in the smooth case \cite{DLZ11}.

Now, the map
\begin{equation*}
\tilde{\omega}_{\tilde{Y}^{\dagger}}^{\bullet}\otimes\mathbb{Q}\rightarrow\mathcal{A}_{\mathbb{Q}}^{\dagger\bullet,0}; \ \ 
\omega\mapsto\omega\wedge\theta
\end{equation*}
induces a quasi-isomorphism $\omega_{\tilde{Y}^{\dagger}}^{\bullet}\otimes\mathbb{Q}\xrightarrow{\sim}\mathcal{A}_{\mathbb{Q}}^{\dagger\bullet}$.  This is (\cite{Gro05}, \S5) in this context, but the argument goes back to \cite{Ste76}. The same argument shows that the map 
\begin{equation*}
W^{\dagger}\tilde{\omega}_{Y/k}^{\bullet}\otimes\mathbb{Q}\rightarrow\mathcal{B}_{\mathbb{Q}}^{\dagger\bullet,0}; \ \ 
\omega\mapsto\omega\wedge\theta
\end{equation*}
induces a quasi-isomorphism $W^{\dagger}\omega_{Y/k}^{\bullet}\otimes\mathbb{Q}\xrightarrow{\sim}\mathcal{B}_{\mathbb{Q}}^{\dagger\bullet}$. As we already noted that $\mathcal{A}_{\mathbb{Q}}^{\dagger\bullet}\cong\mathcal{B}_{\mathbb{Q}}^{\dagger\bullet}$, we conclude that the comparison morphism $(\ref{comparison})$ is a quasi-isomorphism when tensored with $\mathbb{Q}$.
\end{proof}
\begin{Corollary}\label{affine comparison}
Let $Y$ be a semistable affine scheme over $S_{0}$. Then there is a canonical isomorphism 
\begin{equation*}
H_{\emph{\logrig}}^{\ast}(Y/\mathfrak{S}_{0})\cong\mathbb{H}^{\ast}(Y,W^{\dagger}\omega_{Y/k}^{\bullet}\otimes\mathbb{Q})
\end{equation*}
\end{Corollary}
\begin{proof}
We showed that $\mathbb{H}^{\ast}(Y,W^{\dagger}\omega_{Y/k}^{\bullet}\otimes\mathbb{Q})\cong H_{\logmw}^{\ast}(Y/K)$. The comparison between log-Monsky-Washnitzer cohomology and log-rigid cohomology is more or less by definition (see (\cite{Gro05}, \S5.2)).
\end{proof}

\section{Comparison with log-rigid cohomology}\label{main comparison section}

Our aim in this section is to globalise the comparison isomorphism between log-rigid and overconvergent Hyodo-Kato cohomology. We note here that given a $W$-scheme $X$, we shall always write $\hat{X}$ for the formal completion of $X$ along the special fibre, and $\hat{X}_{K}$ for the associated rigid analytic generic fibre.

\begin{Def*}
Let $Y=\spec A$ be a semistable affine scheme over $S_{0}$. A semistable frame for $Y$ is the data of a normal crossings divisor relative to $W$
\begin{equation*}
G=\spec C\hookrightarrow F=\spec B
\end{equation*} 
of affine $W$-schemes where $F$ is smooth over $W$ and
\begin{equation*}
Y\hookrightarrow G_{k}:=G\times_{W}k
\end{equation*} 
is an exact closed $k$-immersion.
\end{Def*}
Note that if $(F,G)$ is a semistable frame for $Y$, then $F$ is a special frame for $Y$ in the sense of \cite{DLZ11} Definition 4.1. Recall from (\cite{DLZ11}, p.253) that the rigid tube $]Y[_{\hat{F}}\subset\hat{F}_{K}$ has a dagger space structure $]Y[^{\dagger}_{\hat{F}}$, which induces a dagger structure $]Y[^{\dagger}_{\hat{G}}$ on $]Y[_{\hat{G}}$. This is functorial in $(Y,F,G)$.
\begin{Def*}
An overconvergent semistable frame for $Y=\spec A$ is the data of a semistable frame $(F=\spec B,G=\spec C)$ for $Y$ along with a homomorphism $\varkappa:C\rightarrow W^{\dagger}(A)$ which lifts the comorphism $C\twoheadrightarrow A$ of the closed $W$-immersion $Y\hookrightarrow G$.    
\end{Def*}

Let $Y=\spec A$ be a semistable affine scheme over $S_{0}$ and suppose that $(F,G,\varkappa)$ is an overconvergent semistable frame for $Y$. Choose an embedding $F\hookrightarrow\mathbb{P}$ into a proper smooth $W$-scheme and write $\overline{G}$ and $\overline{F}$ for the respective closures of $G$ and $F$ inside $\mathbb{P}$. Let $\overline{G}_{k}$ and $\overline{F}_{k}$ be the special fibres of $\overline{G}$ and $\overline{F}$, and let $\overline{Y}$ be the closure of $Y$ inside $\overline{G}_{k}$. Since $\hat{G}_{K}$ is defined in $\hat{F}_{K}$ by overconvergent functions, we can extend the normal crossings divisor $\hat{G}_{K}\hookrightarrow\hat{F}_{K}$ to a normal crossings divisor $V'\hookrightarrow V$ where $V$ is a strict neighbourhood of $]F_{k}[_{\hat{\mathbb{P}}}$ in $]\overline{F}_{k}[_{\hat{\mathbb{P}}}$ and $V'$ is a strict neighbourhood of $]G_{k}[_{\hat{\mathbb{P}}}$ in $]\overline{G}_{k}[_{\hat{\mathbb{P}}}$. We get the following diagram of strict neighbourhoods:
\begin{center}
\begin{tikzpicture}[descr/.style={fill=white,inner sep=1.5pt}]
        \matrix (m) [
            matrix of math nodes,
            row sep=2.5em,
            column sep=0.5em,
            text height=1.5ex, text depth=0.25ex
        ]
        { ]F_{k}[_{\hat{\mathbb{P}}} & \subseteq & V & \subseteq & ]\overline{F}_{k}[_{\hat{\mathbb{P}}} \\
       ]G_{k}[_{\hat{\mathbb{P}}} & \subseteq & V' & \subseteq & ]\overline{G}_{k}[_{\hat{\mathbb{P}}} \\
       ]Y[_{\hat{G}} & \subseteq & \tilde{V}:=V'\cap ]\overline{Y}[_{\hat{\mathbb{P}}} & \subseteq & ]\overline{Y}[_{\hat{\mathbb{P}}} \\};

         \path[overlay,right hook->, font=\scriptsize]
        (m-3-1) edge (m-2-1)
        (m-2-1) edge (m-1-1)
        (m-3-3) edge (m-2-3)
        (m-2-3) edge (m-1-3)
        (m-3-5) edge (m-2-5)
        (m-2-5) edge (m-1-5);
                   
       \end{tikzpicture} 
\end{center} 

In order to define the comparison morphism, we will find it useful to have a rigid analytic description of log-rigid cohomology in terms of sheaves on strict neighbourhoods, in the style of Berthelot. Let $\tilde{\omega}_{\tilde{V}}^{\bullet}$ be the complex given by the restriction of $\Omega_{V}^{\bullet}(\log V')\otimes\mathcal{O}_{V'}$ to $\tilde{V}$, and $\omega_{\tilde{V}}^{\bullet}:=\tilde{\omega}_{\tilde{V}}^{\bullet}/(\tilde{\omega}_{\tilde{V}}^{\bullet-1}\wedge\theta)$ where $\theta=d\log f_{1}+\cdots+d\log f_{s}$ for the functions $f_{i}$ cutting out the normal crossings divisor $\tilde{V}$ in $V$.

Recall that given an abelian sheaf $\mathcal{F}$ on a strict neighbourhood $W$ of $]Y[_{\hat{G}}$ in $]\overline{Y}[_{\widehat{\overline{G}}}$, Berthelot's sheaf of overconvergent sections (\cite{Ber97},\S1.2) is defined to be 
\begin{equation*}
j^{\dagger}\mathcal{F}:=\varinjlim_{V}j_{W,V\ast}j_{W,V}^{-1}\mathcal{F}
\end{equation*}
where the limit is over strict neighbourhoods of $]Y[_{\hat{G}}$ in $W$, and $j_{W,V}:V\hookrightarrow W$ is the inclusion. 

We claim that we have the following Berthelot-style interpretation of log-rigid cohomology: 
\begin{Lemma}\label{rigid dagger comparison}
\begin{equation*}
\mathbb{R}\Gamma(]Y[_{\hat{G}}^{\dagger},\tilde{\omega}_{]Y[_{\hat{G}}^{\dagger}}^{\bullet})=\mathbb{R}\Gamma(\tilde{V},j^{\dagger}\tilde{\omega}_{\tilde{V}}^{\bullet})
\end{equation*}
and 
\begin{equation*}
R\Gamma_{\emph{\logrig}}(Y/\mathfrak{S}_{0}):=\mathbb{R}\Gamma(]Y[_{\hat{G}}^{\dagger},\omega_{]Y[_{\hat{G}}^{\dagger}}^{\bullet})=\mathbb{R}\Gamma(\tilde{V},j^{\dagger}\omega_{\tilde{V}}^{\bullet})
\end{equation*}
\end{Lemma}
\begin{proof}
We shall only prove the first statement since the second is proved using exactly the same argument. 

In order to prove the lemma it suffices to prove that 
\begin{equation*}
\mathbb{R}\Gamma(]Y[_{\hat{G}}^{\dagger},\tilde{\omega}_{]Y[_{\hat{G}}^{\dagger}}^{\bullet})\cong\mathbb{R}\Gamma(]\overline{Y}[_{\widehat{\overline{G}}},j^{\dagger}\tilde{\omega}_{]\overline{Y}[_{\widehat{\overline{G}}}}^{\bullet})
\end{equation*}
Indeed, the right hand side is the same as $\mathbb{R}\Gamma(\tilde{V},j^{\dagger}\tilde{\omega}_{\tilde{V}}^{\bullet})$ by (\cite{Ber97}, 1.2(iv)).

For a coherent sheaf $\mathcal{F}$ on $]\overline{Y}[_{\widehat{\overline{G}}}$ considered as a dagger space with corresponding coherent sheaf $\mathcal{F}'$ on the rigid space $]\overline{Y}[_{\widehat{\overline{G}}}$, let $\tilde{\mathcal{F}}$ be the restriction of $\mathcal{F}$ to the open subspace $]Y[_{\hat{G}}^{\dagger}$. Since the map $]Y[_{\hat{G}}^{\dagger}\xrightarrow{j}]\overline{Y}[_{\widehat{\overline{G}}}$ is affinoid and any section of $\tilde{\mathcal{F}}$ is defined via a neighbourhood of $]Y[_{\hat{G}}^{\dagger}$ in $]\overline{Y}[_{\widehat{\overline{G}}}$, we have a canonical map
\begin{equation*}
Rj_{\ast}\tilde{\mathcal{F}}=j_{\ast}\tilde{\mathcal{F}}\rightarrow j^{\dagger}\mathcal{F}'
\end{equation*} 
If $\mathcal{F}^{\bullet}$ is a complex of coherent sheaves on $]\overline{Y}[_{\widehat{\overline{G}}}$, with corresponding complexes $\mathcal{F}'^{\bullet}$, $\tilde{\mathcal{F}}^{\bullet}$ defined as above, we get a canonical map 
\begin{equation*}
\mathbb{R}\Gamma(]Y[_{\hat{G}}^{\dagger},\tilde{\mathcal{F}}^{\bullet})\rightarrow\mathbb{R}\Gamma(]\overline{Y}[_{\widehat{\overline{G}}},j^{\dagger}\mathcal{F}'^{\bullet})
\end{equation*}
 
Let $\tilde{\omega}_{]\overline{Y}[_{\widehat{\overline{G}}}^{\dagger}}^{\bullet}$ be the log de Rham complex of the log morphism $]\overline{Y}[_{\widehat{\overline{G}}}^{\dagger}\rightarrow(\text{Sp}^{\dagger}K,\{\ast\})$. Since $]\overline{Y}[_{\widehat{\overline{G}}}$ is a partially proper rigid space and $\omega_{]\overline{Y}[_{\widehat{\overline{G}}}^{\dagger}}^{\bullet}$ is a coherent $\mathcal{O}_{]\overline{Y}[_{\widehat{\overline{G}}}^{\dagger}}$-module, we have canonical isomorphisms 
\begin{equation*}
H^{j}(]\overline{Y}[_{\widehat{\overline{G}}},j_{Y}^{\dagger}\tilde{\omega}_{]\overline{Y}[_{\widehat{\overline{G}}}}^{i})\cong H^{j}(]Y[_{\widehat{\overline{G}}}^{\dagger},{\tilde{\omega}_{]\overline{Y}[_{\widehat{\overline{G}}}^{\dagger}}^{i}|}_{]Y[_{\widehat{\overline{G}}}^{\dagger}})=H^{j}(]Y[_{\widehat{\overline{G}}}^{\dagger},\tilde{\omega}_{]Y[_{\widehat{\overline{G}}}^{\dagger}}^{i})
\end{equation*}
for all $i,j$ by \cite{Gro00} Theorem 5.1(a). But $]Y[_{\widehat{\overline{G}}}=]Y[_{\hat{G}}$ and therefore $]Y[_{\widehat{\overline{G}}}^{\dagger}=]Y[_{\hat{G}}^{\dagger}$ too. Hence
\begin{equation*}
H^{j}(]\overline{Y}[_{\widehat{\overline{G}}},j_{Y}^{\dagger}\tilde{\omega}_{]\overline{Y}[_{\widehat{\overline{G}}}}^{i})\cong H^{j}(]Y[_{\hat{G}}^{\dagger},\tilde{\omega}_{]Y[_{\hat{G}}^{\dagger}}^{i})
\end{equation*}
for all $i,j$. Since we have a canonical map
\begin{equation*}
\mathbb{R}\Gamma(]Y[_{\hat{G}}^{\dagger},\tilde{\omega}_{]Y[_{\hat{G}}^{\dagger}}^{\bullet})\rightarrow\mathbb{R}\Gamma(]\overline{Y}[_{\widehat{\overline{G}}},j^{\dagger}\tilde{\omega}_{]\overline{Y}[_{\widehat{\overline{G}}}}^{\bullet})
\end{equation*}
we conclude from the first hypercohomology spectral sequence that
\begin{equation*}
\mathbb{R}\Gamma(]\overline{Y}[_{\widehat{\overline{G}}},j_{Y}^{\dagger}\tilde{\omega}_{]\overline{Y}[_{\widehat{\overline{G}}}}^{\bullet})\cong\mathbb{R}\Gamma(]Y[_{\hat{G}}^{\dagger},\tilde{\omega}_{]Y[_{\hat{G}}^{\dagger}}^{\bullet})
\end{equation*}
as required.

The second statement is proved with exactly the same argument, but where one instead considers the log de Rham complex with respect to the base $(K,\mathbb{N})$. 
\end{proof}

In \cite{DLZ11} \S4, there is constructed an explicit fundamental system of strict affinoid neighbourhoods $V_{\lambda,\eta}$ of $]Y[_{\hat{G}}$ in $]\overline{Y}[_{\hat{\mathbb{P}}}$ (here $0<\lambda,\eta<1$), canonical morphisms (\cite{DLZ11}, pp 251)
\begin{equation*}
\Gamma(V_{\lambda,\eta},j^{\dagger}\mathcal{O}_{V_{\lambda,\eta}})\rightarrow W^{\dagger}(A)\otimes\mathbb{Q}
\end{equation*}
and therefore morphisms
\begin{equation*}
\Gamma(\tilde{V},j^{\dagger}\mathcal{O}_{\tilde{V}})\rightarrow W^{\dagger}(A)\otimes\mathbb{Q}
\end{equation*}
for any strict neighbourhood $\tilde{V}$ of $]Y[_{\hat{G}}$ in $]\overline{Y}[_{\hat{\mathbb{P}}}$.  The universal property of the de Rham complex then gives a map
\begin{equation*}
\Gamma(\tilde{V},j^{\dagger}\tilde{\omega}_{\tilde{V}}^{\bullet})\rightarrow u_{\ast}W^{\dagger}\Omega_{U/k}^{\bullet}[\theta]/(\theta^{2})\otimes\mathbb{Q}
\end{equation*}
where $u:U\hookrightarrow Y$ is the smooth locus of $Y=\spec A$, and this clearly factors through
\begin{equation}\label{map 1}
\Gamma(\tilde{V},j^{\dagger}\tilde{\omega}_{\tilde{V}}^{\bullet})\rightarrow W^{\dagger}\tilde{\omega}_{A/k}^{\bullet}\otimes\mathbb{Q}
\end{equation}
The argument used after \cite{DLZ11} (4.28) can be used verbatim to show that this factors through a morphism
\begin{equation}\label{tilde morphism}
\mathbb{R}\Gamma(\tilde{V},j^{\dagger}\tilde{\omega}_{\tilde{V}}^{\bullet})\rightarrow W^{\dagger}\tilde{\omega}_{A/k}^{\bullet}\otimes\mathbb{Q}
\end{equation}
Indeed, $\tilde{V}$ contains some $V_{\lambda,\eta}$ and we can consider the restriction
\begin{equation*}
\mathbb{R}\Gamma(\tilde{V},j^{\dagger}\tilde{\omega}_{\tilde{V}}^{\bullet})\rightarrow\mathbb{R}\Gamma(V_{\lambda,\eta},j^{\dagger}\tilde{\omega}_{V_{\lambda,\eta}}^{\bullet})
\end{equation*} 
Given any strict neighbourhood $\tilde{V}'$ of $]Y[_{\hat{G}}$ in $\tilde{V}$, let us write $\alpha_{\tilde{V}'}:\tilde{V}'\cap V_{\lambda,\eta}\hookrightarrow V_{\lambda,\eta}$ for the inclusion. Then by the definition of $j^{\dagger}$ we have 
\begin{equation*}
j^{\dagger}\tilde{\omega}_{V_{\lambda,\eta}}^{\bullet}=\varinjlim_{\tilde{V}'}\alpha_{\tilde{V}'\ast}\tilde{\omega}_{\tilde{V}'\cap V_{\lambda,\eta}}^{\bullet}
\end{equation*}
where the direct limit runs over all strict neighbourhoods  $\tilde{V}'$ of $]Y[_{\hat{G}}$ in $\tilde{V}$. Therefore 
\begin{equation*}
\mathbb{R}\Gamma(V_{\lambda,\eta},j^{\dagger}\tilde{\omega}_{V_{\lambda,\eta}}^{\bullet})=\mathbb{R}\Gamma(V_{\lambda,\eta},\varinjlim_{\tilde{V}'}\alpha_{\tilde{V}'\ast}\tilde{\omega}_{\tilde{V}'\cap V_{\lambda,\eta}}^{\bullet})\cong\varinjlim_{\tilde{V}'}\mathbb{R}\Gamma(V_{\lambda,\eta},\alpha_{\tilde{V}'\ast}\tilde{\omega}_{\tilde{V}'\cap V_{\lambda,\eta}}^{\bullet})
\end{equation*}
where the isomorphism is by the quasicompactness of $V_{\lambda,\eta}$. Now for each $\tilde{V}'$ one can find a $\lambda'$ such that $V_{\lambda',\eta}$ is a strict affinoid neighbourhood of $]Y[_{\hat{G}}$ in $\tilde{V}'\cap V_{\lambda,\eta}$. The restriction to the affinoids $V_{\lambda',\eta}$ gives a map
\begin{equation*}
\mathbb{R}\Gamma(V_{\lambda,\eta},j^{\dagger}\tilde{\omega}_{V_{\lambda,\eta}}^{\bullet})\rightarrow\varinjlim_{\lambda'}\mathbb{R}\Gamma(V_{\lambda',\eta},\tilde{\omega}_{V_{\lambda',\eta}}^{\bullet})\cong\varinjlim_{\lambda'}\Gamma(V_{\lambda',\eta},\tilde{\omega}_{V_{\lambda',\eta}}^{\bullet})\rightarrow W^{\dagger}\tilde{\omega}_{A/k}^{\bullet}\otimes\mathbb{Q}
\end{equation*}
where the isomorphism is because each $V_{\lambda',\eta}$ is affinoid and the last map is induced by the morphisms $\Gamma(V_{\lambda',\eta},\mathcal{O}_{V_{\lambda',\eta}})\rightarrow W^{\dagger}(A)\otimes\mathbb{Q}$ constructed in (\cite{DLZ11}, pp.251). Precomposing with the restriction $\mathbb{R}\Gamma(\tilde{V},j^{\dagger}\tilde{\omega}_{\tilde{V}}^{\bullet})\rightarrow\mathbb{R}\Gamma(V_{\lambda,\eta},j^{\dagger}\tilde{\omega}_{V_{\lambda,\eta}}^{\bullet})$ then gives the desired morphism.

If $f_{1},\ldots, f_{s}$ define the normal crossings divisor $\tilde{V}$ in $V$ then $f_{1}\cdots f_{s}=0$, and hence $\bar{f}_{1}\cdots\bar{f}_{s}=0$. Therefore $d\log[\bar{f}_{1}]+\cdots+d\log[\bar{f}_{s}]=0$ and the morphism (\ref{tilde morphism}) induces a morphism
\begin{equation}\label{morphism}
R\Gamma_{\logrig}(Y/\mathfrak{S}_{0})=\mathbb{R}\Gamma(\tilde{V},j^{\dagger}\omega_{\tilde{V}}^{\bullet})\rightarrow W^{\dagger}\omega_{A/k}^{\bullet}\otimes\mathbb{Q}
\end{equation}

\begin{Proposition}\label{isomorphisms and independence}
The morphisms (\ref{tilde morphism}) and (\ref{morphism}) for overconvergent semistable frames are isomorphisms in the derived category and do not depend on the choice of overconvergent semistable frame for $Y$.
\end{Proposition}
\begin{proof}
We first prove the independence assertion. Let $(F,G,\varkappa)$ and $(F',G',\varkappa')$ be two overconvergent semistable frames for $Y$. Let 
\begin{equation*}
F\xleftarrow{\text{pr}_{1}}F\times_{W}F'\xrightarrow{\text{pr}_{2}}F'
\end{equation*}
be the projections. Then the product 
\begin{equation*}
(F'',G'',\varkappa''):=(F\times_{W}F', \text{pr}_{1}^{-1}(G)+\text{pr}_{2}^{-1}(G'),\varkappa\otimes\varkappa')
\end{equation*}
is another overconvergent semistable frame for $Y$. Choose strict neighbourhoods $\tilde{V}$, $\tilde{V}'$ and $\tilde{V}''$ such that $\tilde{V}''$ is sent to $\tilde{V}$ and $\tilde{V}'$ by the respective projections. By functoriality, the projections induce diagrams
\begin{center}
\begin{tikzpicture}[descr/.style={fill=white,inner sep=1.5pt}]
        \matrix (m) [
            matrix of math nodes,
            row sep=3.5em,
            column sep=2.5em,
            text height=1.5ex, text depth=0.25ex
        ]
        { \mathbb{R}\Gamma(\tilde{V},j^{\dagger}\tilde{\omega}_{\tilde{V}}^{\bullet}) & \mathbb{R}\Gamma(\tilde{V}'',j^{\dagger}\tilde{\omega}_{\tilde{V}''}^{\bullet})  & \mathbb{R}\Gamma(\tilde{V}',j^{\dagger}\tilde{\omega}_{\tilde{V}'}^{\bullet})  \\
       \ & W^{\dagger}\tilde{\omega}_{A/k}^{\bullet}\otimes\mathbb{Q} & \ \\};

         \path[overlay,->, font=\scriptsize]
        (m-1-1) edge node [above] {$\text{pr}_{1}^{\ast}$} (m-1-2)
        (m-1-3) edge node [above] {$\text{pr}_{2}^{\ast}$} (m-1-2)
        (m-1-1) edge (m-2-2)
        (m-1-3) edge (m-2-2)
        (m-1-2) edge (m-2-2);
                   
       \end{tikzpicture} 
\end{center} 
and
\begin{center}
\begin{tikzpicture}[descr/.style={fill=white,inner sep=1.5pt}]
        \matrix (m) [
            matrix of math nodes,
            row sep=3.5em,
            column sep=2.5em,
            text height=1.5ex, text depth=0.25ex
        ]
        { \mathbb{R}\Gamma(\tilde{V},j^{\dagger}\omega_{\tilde{V}}^{\bullet}) & \mathbb{R}\Gamma(\tilde{V}'',j^{\dagger}\omega_{\tilde{V}''}^{\bullet})  & \mathbb{R}\Gamma(\tilde{V}',j^{\dagger}\omega_{\tilde{V}'}^{\bullet})  \\
       \ & W^{\dagger}\omega_{A/k}^{\bullet}\otimes\mathbb{Q} & \ \\};

         \path[overlay,->, font=\scriptsize]
        (m-1-1) edge node [above] {$\text{pr}_{1}^{\ast}$} (m-1-2)
        (m-1-3) edge node [above] {$\text{pr}_{2}^{\ast}$} (m-1-2)
        (m-1-1) edge (m-2-2)
        (m-1-3) edge (m-2-2)  
        (m-1-2) edge (m-2-2);
                   
       \end{tikzpicture} 
\end{center} 
and we see therefore that the morphisms (\ref{tilde morphism}) and (\ref{morphism}) do not depend on the choice of overconvergent semistable frame for $Y$. 

To prove that the morphisms are isomorphisms, since we have already shown independence, we may as well work with the log-Monsky-Washnitzer frame $(\tilde{Z},\tilde{Y},\varkappa)$ as in \S\ref{logmw section} for the overconvergent semistable frame for $Y$ (remember that $Y$ is affine). We then conclude by Theorem \ref{rational quasi-isomorphisms}. 
\end{proof}

\begin{Theorem}\label{main comparison}
Let $Y$ be a quasi-projective semistable scheme over $S_{0}$. Then the overconvergent Hyodo-Kato complex computes the log-rigid cohomology of $Y$:
\begin{equation*}
R\Gamma_{\emph{\logrig}}(Y/\mathfrak{S}_{0})\cong\mathbb{R}\Gamma(Y,W^{\dagger}\omega_{Y/k}^{\bullet}\otimes\mathbb{Q})
\end{equation*}
\end{Theorem}
\begin{proof}
Let $Y=\bigcup_{i\in I}Y_{i}$ be an open covering, and for $J=\{i_{0},\ldots,i_{t}\}\subset I$, let $Y_{J}:=Y_{i_{0}}\cap\cdots\cap Y_{i_{t}}$. By choosing a possibly finer covering, we may assume that $Y_{J}=\spec A_{J}$ is affine and that $A_{J}=(A_{i_{0}})_{\bar{g}}$ for some element $\bar{g}\in A_{i_{0}}$, where $Y_{i_{0}}=\spec A_{i_{0}}$. It is here that we use the quasi-projectivity hypothesis (compare the argument in (\cite{DLZ11}, Def. 4.33) and the subsequent discussion). For each $i\in J$, choose a smooth affine $k$-scheme $X_{i}=\spec B_{i}$ such that $Y_{i}$ is a normal crossings divisor in $X_{i}$. We may assume that each $X_{i}$ is standard smooth in the sense of (\cite{DLZ11}, Def. 4.33)). Let $F_{i}$ be a smooth affine $W$-scheme lifting $X_{i}$, which is again standard smooth, and let $Z_{i}$ be a lifting over $W$ of $Y_{i}$ which is a normal crossings divisor in $F_{i}$ (compare with (\cite{Kat96}, Prop. 11.3)). 

Now let $Z_{i_{0}}=\spec\tilde{A}_{i_{0}}$ and $Z'_{i_{0}}=\spec(\tilde{A}_{i_{0}})_{g}$ for some lifting $g$ of $\bar{g}$, and let $F_{i_{0}}=\spec\tilde{B}_{i_{0}}$ and $F'_{i_{0}}=\spec(\tilde{B}_{i_{0}})_{f}$ for some lifting $f$ of $g$. Set 
\begin{equation*}
E:=\prod_{\stackrel{i\in J}{i\neq i_{0}}}F_{i}
\end{equation*}
Then, by the strong fibration theorem, the special frames $(Y_{J}, F_{i_{0}}\times E)$ and $(Y_{J}, F'_{i_{0}}\times E)$, have isomorphic dagger spaces. See the argument in the proof of (\cite{DLZ11}, Prop. 4.35). Since $E$ is standard smooth, we can choose an \'{e}tale map $E\rightarrow\mathbb{A}_{W}^{n}$ for some $n$. Again by the strong fibration theorem, the dagger spaces associated to $(Y_{J}, F'_{i_{0}}\times E)$ and $(Y_{J}, F'_{i_{0}}\times\mathbb{A}_{W}^{n})$, are isomorphic. By the coordinate change argument in the proof of (\cite{DLZ11}, Prop. 4.35), we may assume that the map $Y_{J}\rightarrow\mathbb{A}_{W}^{n}$ factors through the zero section $\spec k\rightarrow\mathbb{A}_{W}^{n}$. Hence the dagger space associated to $(Y_{J}, F'_{i_{0}}\times\mathbb{A}_{W}^{n})$ is isomorphic to $Q\times\breve{D}^{n}$, where $\breve{D}$ is the open unit dagger disk and $Q$ is the dagger space associated to the special frame $(Y_{J},F'_{i_{0}})$. Using the notation of (\cite{DLZ11}, p.252), we write the dagger space associated to $(Y_{J},F'_{i_{0}}\times\mathbb{A}_{W}^{n})$ as $Q\times\breve{D}^{n}=:]Y_{J}[_{\widehat{F_{i_{0}}'\times\mathbb{A}_{W}^{n}}}^{\dagger}$, where $\widehat{F_{i_{0}}'\times\mathbb{A}_{W}^{n}}$ denotes the weak formal completion of $F_{i_{0}}'\times\mathbb{A}_{W}^{n}$ along $p$.

Now consider the embeddings
\begin{equation*}
Y_{J}\hookrightarrow\prod_{i\in J}Z_{i}\hookrightarrow Z_{i_{0}}\times\prod_{\stackrel{j\in J}{j\neq i_{0}}}F_{j}\hookrightarrow F_{i_{0}}\times E
\end{equation*}
and 
\begin{equation*}
Y_{J}\hookrightarrow Z'_{i_{0}}\times\prod_{\stackrel{i\in J}{i\neq i_{0}}}Z_{i}\hookrightarrow\sum_{i\in J}(Z'_{i}\times\prod_{\stackrel{j\in J}{j\neq i}}F'_{j})\hookrightarrow F'_{i_{0}}\times E
\end{equation*}
where
\begin{equation*}
Z'_{i}:=\begin{cases} 
      Z_{i} & \text{if }i\neq i_{0} \\
      Z'_{i_{0}} & \text{if }i=i_{0}  
   \end{cases}
\end{equation*}
and likewise for $F'_{i}$. Note that
\begin{equation*}
D_{J}:=\sum_{i\in J}(Z'_{i}\times\prod_{\stackrel{j\in J}{j\neq i}}F'_{j})=(Z'_{i_{0}}\times\prod_{\stackrel{i\in J}{i\neq i_{0}}}F'_{i})+\sum_{\stackrel{i\in J}{i\neq i_{0}}}(F'_{i_{0}}\times Z_{i}\times\prod_{\stackrel{j\in J}{j\neq i,i_{0}}}F_{i})
\end{equation*}
is a normal crossings divisor in $F'_{i_{0}}\times E$, and $]Y_{J}[_{\hat{D}_{J}}^{\dagger}$ is a normal crossings divisor in  $]Y_{J}[_{\widehat{F_{i_{0}}'\times E}}^{\dagger}$. Applying the strong fibration theorem and coordinate change argument as above, we get a commutative diagram of dagger spaces
\begin{center}
\begin{tikzpicture}[descr/.style={fill=white,inner sep=1.5pt}]
        \matrix (m) [
            matrix of math nodes,
            row sep=2.5em,
            column sep=2.5em,
            text height=1.5ex, text depth=0.25ex
        ]
        { ]Y_{J}[_{\hat{D}_{J}}^{\dagger} & ]Y_{J}[_{\widehat{F_{i_{0}}'\times E}}^{\dagger} & ]X_{J}[_{\widehat{F_{i_{0}}'\times E}}^{\dagger} \\
          M_{J} & Q\times\breve{D}^{n} & \tilde{Q}\times\breve{D}^{n} \\
        };

         \path[overlay,right hook->, font=\scriptsize]
        (m-1-1) edge (m-1-2)
        (m-1-2) edge (m-1-3)
        (m-2-1) edge (m-2-2)
        (m-2-2) edge (m-2-3);
        
        \path[overlay,->, font=\scriptsize]
        (m-1-1) edge node [right] {$\wr$} (m-2-1);       
        
        \draw[shorten <=0.2cm,->, font=\scriptsize]
        (m-1-2) edge node [right] {$\wr$} (m-2-2)
        (m-1-3) edge node [right] {$\wr$} (m-2-3);

       \end{tikzpicture} 
\end{center} 
where the dagger space $M_{J}$, which is a normal crossings divisor in $Q\times\breve{D}^{n}$, is a sum of normal crossings divisors of the following form:-
\\
\par
(a) $]Y_{J}[_{\hat{Z}'_{i_{0}}}^{\dagger}\times\breve{D}^{n}$, where $]Y_{J}[_{\hat{Z}'_{i_{0}}}^{\dagger}$ is a normal crossings divisor in $Q$
\\
\par
(b) $Q\times\breve{D}^{n}(m)$, where $\breve{D}^{n}(m)$ is the divisor in $\breve{D}^{n}$ corresponding to 
\begin{equation*}
\text{Sp } K\langle T_{1},\ldots, T_{n}\rangle^{\dagger}/(T_{1}\cdots T_{m})
\end{equation*}

Let $\omega_{M_{J}}^{\bullet}$ denote the logarithmic de Rham complex on the normal crossings divisor $M_{J}$ in $Q\times\breve{D}^{n}$, as defined in \cite{Gro05}. We rewrite the comparison morphism defined in (\ref{morphism}) in terms of dagger spaces using Lemma \ref{rigid dagger comparison}. Then for the case (a) we have a map
\begin{align*}
\Gamma(]Y_{J}[_{\hat{Z}'_{i_{0}}}^{\dagger}\times\breve{D}^{n},\omega_{]Y_{J}[_{\hat{Z}'_{i_{0}}}^{\dagger}\times\breve{D}^{n}}^{\bullet})
& =\Gamma(]Y_{J}[_{\hat{Z}'_{i_{0}}}^{\dagger}\times\breve{D}^{n},\omega_{]Y_{J}[_{\hat{Z}'_{i_{0}}}^{\dagger}}^{\bullet}\otimes\Omega_{\breve{D}^{n}}^{\bullet})\\
& \rightarrow\Gamma(]Y_{J}[_{\hat{Z}'_{i_{0}}}^{\dagger},\omega_{]Y_{J}[_{\hat{Z}'_{i_{0}}}^{\dagger}}^{\bullet})\rightarrow W^{\dagger}\omega_{A_{J}/k}^{\bullet}\otimes\mathbb{Q}
\end{align*}
where $\Omega_{\breve{D}^{n}}^{\bullet}$ is the usual (non-logarithmic) de Rham complex on $\breve{D}^{n}$, and where the first map is the projection and the second comes from the comparison between the log-Monsky-Washnitzer complex and overconvergent Hyodo-Kato complex constructed in (\ref{comparison}). For the case (b) we have a map
\resizebox{1.0\linewidth}{!}{
  \begin{minipage}{\linewidth}
\begin{align*}
\Gamma(Q\times\breve{D}^{n}(m),\omega_{Q\times\breve{D}^{n}(m)}^{\bullet})
& =\Gamma(Q\times\breve{D}^{n}(m),\omega_{]Y_{J}[_{\hat{F}'_{i_{0}}}^{\dagger}\times\breve{D}^{n}(m)}^{\bullet}) \\
& =\Gamma(Q\times\breve{D}^{n}(m),\Omega_{]Y_{J}[_{\hat{F}'_{i_{0}}}^{\dagger}}^{\bullet}\otimes\omega_{\breve{D}^{n}(m)}^{\bullet}) \\
& \rightarrow\Gamma(]Y_{J}[_{\hat{F}'_{i_{0}}}^{\dagger},\Omega_{]Y_{J}[_{\hat{F}'_{i_{0}}}^{\dagger}}^{\bullet})\rightarrow W^{\dagger}\Omega_{A_{J}/k}^{\bullet}\otimes\mathbb{Q}\rightarrow W^{\dagger}\omega_{A_{J}/k}^{\bullet}\otimes\mathbb{Q}
\end{align*}
\end{minipage}
}
where $\Omega_{]Y_{J}[_{\hat{F}'_{i_{0}}}^{\dagger}}^{\bullet}$ is the usual de Rham complex on $]Y_{J}[_{\hat{F}'_{i_{0}}}^{\dagger}$ and the first map is again the projection. Let $\text{sp}: \ ]Y_{J}[_{\hat{D}_{J}}^{\dagger}=M_{J}\rightarrow Y_{J}$ be the specialisation map. Then by the argument (\cite{DLZ11}, 4.32), we have a local version of the above morphisms and get morphisms of complexes of Zariski sheaves on $Y_{J}$
\begin{equation*}
\text{sp}_{\ast}\omega_{]Y_{J}[_{\hat{Z}'_{i_{0}}}^{\dagger}\times\breve{D}^{n}}^{\bullet}\rightarrow W^{\dagger}\omega_{Y_{J}/k}^{\bullet}\otimes\mathbb{Q}
\end{equation*}
and
\begin{equation*}
\text{sp}_{\ast}\omega_{Q\times\breve{D}^{n}(m)}^{\bullet}\rightarrow W^{\dagger}\omega_{Y_{J}/k}^{\bullet}\otimes\mathbb{Q}
\end{equation*}
which give rise to a morphism
\begin{equation}\label{affine pieces}
\text{sp}_{\ast}\omega_{]Y_{J}[_{\hat{D}_{J}}^{\dagger}}^{\bullet}=\text{sp}_{\ast}\omega_{M_{J}}^{\bullet}\rightarrow W^{\dagger}\omega_{Y_{J}/k}^{\bullet}\otimes\mathbb{Q}
\end{equation}
into the overconvergent Hyodo-Kato complex (tensored with $\mathbb{Q}$) of $Y_{J}$. For the convenience of the reader we recall the argument in (\cite{DLZ11}, pp 253 and 254). We use the notation as above. Let $Y_{J}=\spec A_{J}$, with $A_{J}=(A_{i_{0}})_{\bar{g}}$, $Z'_{i_{0}}=\spec(\tilde{A}_{i_{0}})_{g}$ with $g$ a lifting of $\bar{g}$, and $F'_{i_{0}}=\spec(\tilde{B}_{i_{0}})_{f}$ with $f$ a lifting of $g$. Let $U=\spec(A_{J})_{\bar{h}}$, $h$ a lifting of $\bar{h}$ in $(\tilde{A}_{i_{0}})_{g}$ and $\tilde{h}$ a lifting of $h$ in $(\tilde{B}_{i_{0}})_{f}$. Let $Z''_{i_{0}}=\spec((\tilde{A}_{i_{0}})_{g})_{h}$ and $F''_{i_{0}}=\spec((\tilde{B}_{i_{0}})_{f})_{\tilde{h}}$. Then $]U[_{\hat{Z}'_{i_{0}}}$ is open in $]Y_{J}[_{\hat{Z}'_{i_{0}}}$ and $]U[_{\hat{F}'_{i_{0}}}$ is open in $]Y_{J}[_{\hat{F}'_{i_{0}}}$, hence $]U[_{\hat{Z}'_{i_{0}}}$ and $]U[_{\hat{F}'_{i_{0}}}$ inherit dagger space structures from $]Y_{J}[^{\dagger}_{\hat{Z}'_{i_{0}}}$ and $]Y_{J}[^{\dagger}_{\hat{F}'_{i_{0}}}$. By the strong fibration theorem applied in the same way as in the section following (\cite{DLZ11}, 4.32), the dagger spaces $]U[^{\dagger}_{\hat{Z}'_{i_{0}}}$ and $]U[^{\dagger}_{\hat{Z}''_{i_{0}}}$ resp. $]U[^{\dagger}_{\hat{F}'_{i_{0}}}$ and $]U[^{\dagger}_{\hat{F}''_{i_{0}}}$ coincide. This induces the two morphisms 
\begin{equation*}
\text{sp}_{\ast}\omega_{]Y_{J}[_{\hat{Z}'_{i_{0}}}^{\dagger}\times\breve{D}^{n}}^{\bullet}\rightarrow W^{\dagger}\omega_{Y_{J}/k}^{\bullet}\otimes\mathbb{Q}
\end{equation*}
and
\begin{equation*}
\text{sp}_{\ast}\omega_{Q\times\breve{D}^{n}(m)}^{\bullet}\rightarrow W^{\dagger}\omega_{Y_{J}/k}^{\bullet}\otimes\mathbb{Q}
\end{equation*}
which combine to give the morphism (\ref{affine pieces}).

We claim, in analogy to (\cite{DLZ11}, Cor. 4.38) that the canonical morphisms
\begin{equation*}
\text{sp}_{\ast}\omega_{M_{J}}^{\bullet}\rightarrow\mathbb{R}\text{sp}_{\ast}\omega_{M_{J}}^{\bullet}
\end{equation*}
are quasi-isomorphisms. For now we will assume this claim; the proof is postponed until Proposition \ref{crucial quasi-isomorphism}. Then (\ref{affine pieces}) together with the claim gives a morphism
\begin{equation}\label{R affine pieces}
\mathbb{R}\text{sp}_{\ast}\omega_{M_{J}}^{\bullet}\rightarrow W^{\dagger}\omega_{Y_{J}/k}^{\bullet}\otimes\mathbb{Q}
\end{equation}
Let $U=\spec C$ be an affine open in $Y_{J}$. By applying $\mathbb{R}\Gamma(U,-)$ to (\ref{R affine pieces}) we get from Lemma \ref{rigid dagger comparison} and Proposition \ref{isomorphisms and independence}  an isomorphism
\begin{equation*}
\mathbb{R}\Gamma(U,\mathbb{R}\text{sp}_{\ast}\omega_{]\text{sp}^{-1}(U)[^{\dagger}_{\hat{D}_{J}}}^{\bullet})\xrightarrow{\sim}\mathbb{R}\Gamma(U,W^{\dagger}\omega_{U/k}^{\bullet}\otimes\mathbb{Q})\simeq W^{\dagger}\omega_{C/k}^{\bullet}\otimes\mathbb{Q}
\end{equation*}
where the latter isomorphism follows from the fact that $H^{i}(U,W^{\dagger}\omega_{U/k}^{r})=0$ for all $r\geq 0,i>0$. This is the semistable analogue of (\cite{DLZ11}, Prop. 1.2(b)) and is derived from the smooth case by considering the graded quotients of the weight filtration on the \v{C}ech complex associated to $W^{\dagger}\omega_{C/k}^{r}$ as in the proof of Propositon \ref{Cech}. The above isomorphism shows that (\ref{R affine pieces}) is already an isomorphism.
 
Now, as we range through the subsets $J\subset I$, we get an augmented simplicial $k$-scheme $\theta:Y_{\bullet}:=\{Y_{J}\}_{J\subset I}\rightarrow Y$. Let
\begin{equation*}
\tilde{D}_{J}:=\sum_{i\in J}(Z_{i}\times\prod_{\stackrel{j\in J}{j\neq i}}F_{j})
\end{equation*}
which is a normal crossings divisor in $F_{i_{0}}\times E$. Again by the strong fibration theorem the dagger tubes $]Y_{J}[_{\hat{\tilde{D}}_{J}}^{\dagger}$ and $]Y_{J}[_{\hat{D}_{J}}^{\dagger}$ are isomorphic. We get a simplicial object of special frames $\left\{(Y_{J},\tilde{D}_{J})\right\}_{J\subset I}$, and this gives rise to a simplicial object of dagger spaces
\begin{equation*}
M_{\bullet}:=\left\{]Y_{J}[_{\hat{\tilde{D}}_{J}}^{\dagger}\right\}_{J\subset I}=\left\{M_{J}\right\}_{J\subset I}
\end{equation*}
The quasi-isomorphisms (\ref{R affine pieces}) glue to give a quasi-isomorphism of simplicial complexes on $Y_{\bullet}$
\begin{equation}\label{simplicial}
\mathbb{R}\text{sp}_{\ast}\omega_{M_{\bullet}}^{\bullet}\xrightarrow{\sim} W^{\dagger}\omega_{Y_{\bullet}/k}^{\bullet}\otimes\mathbb{Q}
\end{equation}
Therefore 
\begin{equation*}
\mathbb{R}\theta_{\ast}\mathbb{R}\text{sp}_{\ast}\omega_{M_{\bullet}}^{\bullet}\cong \mathbb{R}\theta_{\ast}W^{\dagger}\omega_{Y_{\bullet}/k}^{\bullet}\otimes\mathbb{Q}\cong W^{\dagger}\omega_{Y/k}^{\bullet}\otimes\mathbb{Q}
\end{equation*}
and we deduce that
\begin{equation*}
R\Gamma_{\logrig}(Y/\mathfrak{S}_{0})=\mathbb{R}\Gamma(Y,\mathbb{R}\theta_{\ast}\mathbb{R}\text{sp}_{\ast}\omega_{M_{\bullet}}^{\bullet})\cong\mathbb{R}\Gamma(Y,W^{\dagger}\omega_{Y/k}^{\bullet}\otimes\mathbb{Q})
\end{equation*}
as desired.
\end{proof}

It therefore remains to prove the following proposition:
\begin{Proposition}\label{crucial quasi-isomorphism}
Let $M_{J}$ be the dagger space considered in the proof of Theorem \ref{main comparison}. Then the canonical morphism
\begin{equation*}
\text{sp}_{\ast}\omega_{M_{J}}^{\bullet}\rightarrow\mathbb{R}\text{sp}_{\ast}\omega_{M_{J}}^{\bullet}
\end{equation*}
is a quasi-isomorphism.
\end{Proposition}

The proof will occupy us for the rest of the section. By using the Mayer-Vietoris exact sequence, it is easy to see that it suffices to prove the proposition separately for the two cases (a) and (b) above. That is, it suffices to prove that
\begin{equation*}
\text{sp}_{\ast}\omega_{]Y_{J}[_{\hat{Z}_{i_{0}}}^{\dagger}\times\breve{D}^{n}}^{\bullet}\rightarrow\mathbb{R}\text{sp}_{\ast}\omega_{]Y_{J}[_{\hat{Z}_{i_{0}}}^{\dagger}\times\breve{D}^{n}}^{\bullet}
\end{equation*} 
and 
\begin{equation*}
\text{sp}_{\ast}\omega_{Q\times\breve{D}^{n}(m)}^{\bullet}\rightarrow\mathbb{R}\text{sp}_{\ast}\omega_{Q\times\breve{D}^{n}(m)}^{\bullet}
\end{equation*} 
are quasi-isomorphisms. We recall from the proof of Theorem \ref{main comparison} that we have
\begin{equation*}
\omega_{]Y_{J}[_{\hat{Z}_{i_{0}}}^{\dagger}\times\breve{D}^{n}}^{\bullet}=\omega_{]Y_{J}[_{\hat{Z}'_{i_{0}}}^{\dagger}}^{\bullet}\otimes\Omega_{\breve{D}^{n}}^{\bullet}
\end{equation*} 
and 
\begin{equation*}
\omega_{Q\times\breve{D}^{n}(m)}^{\bullet} =\Omega_{]Y_{J}[_{\hat{F}'_{i_{0}}}^{\dagger}}^{\bullet}\otimes\omega_{\breve{D}^{n}(m)}^{\bullet}
\end{equation*} 

The proof for case (a) is easy. Indeed, in Proposition 4.37 and Corollary 4.38 of \cite{DLZ11}, it is not needed that $Q$ is a smooth affinoid dagger space. What is needed is that $\Omega_{Q}^{p}$ is a locally free $\mathcal{O}_{Q}$-module and that $H^{i}(Q,\Omega_{Q}^{p})$ vanishes for $i>0$ (Tate-acyclicity). Both properties hold for the locally free $(\tilde{A}_{i_{0}})_{g}^{\dagger}\otimes\mathbb{Q}$-module $\omega_{]Y_{J}[_{\hat{Z}_{i_{0}}}^{\dagger}}^{p}$ as well, indeed $H^{i}(]Y_{J}[_{\hat{Z}_{i_{0}}}^{\dagger},\omega_{]Y_{J}[_{\hat{Z}_{i_{0}}}^{\dagger}}^{p})=0$ for $i>0$ because $]Y_{J}[_{\hat{Z}_{i_{0}}}^{\dagger}$ is affinoid. Hence we can replace $Q$ by $]Y_{J}[_{\hat{Z}_{i_{0}}}^{\dagger}$ and $\Omega_{Q}^{p}$ by $\omega_{]Y_{J}[_{\hat{Z}_{i_{0}}}^{\dagger}}^{p}$ in the proofs of Proposition 4.37, Corollary 4.38 and Lemmas 4.44$-$4.47 in \cite{DLZ11} to obtain the desired quasi-isomorphism
\begin{equation*}
\mathbb{R}\text{sp}_{\ast}\omega_{]Y_{J}[_{\hat{Z}_{i_{0}}}^{\dagger}\times\breve{D}^{n}}^{\bullet}\cong\text{sp}_{\ast}\omega_{]Y_{J}[_{\hat{Z}_{i_{0}}}^{\dagger}\times\breve{D}^{n}}^{\bullet}
\end{equation*} 

Now we will treat case (b), which is more subtle. Since $Q$ is an open subspace in the smooth affinoid dagger space $\tilde{Q}$, it is enough to show that
\begin{equation*}
\mathbb{R}\text{sp}_{\ast}\omega_{\tilde{Q}\times\breve{D}^{n}(m)}^{\bullet}\cong\text{sp}_{\ast}\omega_{\tilde{Q}\times\breve{D}^{n}(m)}^{\bullet}
\end{equation*}
Note that we have
\begin{equation*}
\omega_{\tilde{Q}\times\breve{D}^{n}(m)}^{\bullet} =\Omega_{\tilde{Q}}^{\bullet}\otimes\omega_{\breve{D}^{n}(m)}^{\bullet}
\end{equation*} 

We have analogues of Lemma 4.45 and Lemma 4.47 of \cite{DLZ11}:
\begin{Lemma}\label{closed}
Let $A$ be a smooth dagger algebra and $Q=\text{Sp }A$ the associated affinoid dagger space, and $D^{n}(m)$ the normal crossings divisor on the closed unit dagger $n$-ball $D^{n}$ associated to 
\begin{equation*}
\text{Sp } K\langle T_{1},\ldots, T_{n}\rangle^{\dagger}/(T_{1}\cdots T_{m})
\end{equation*}
Let
\begin{equation*}
\Lambda_{n}:=(A\otimes_{K}\omega_{D^{n}(m)}^{0}\rightarrow A\otimes_{K}\omega_{D^{n}(m)}^{1}\rightarrow A\otimes_{K}\omega_{D^{n}(m)}^{2}\rightarrow\cdots)
\end{equation*}
be the complex with obvious differential. Let 
\begin{equation*}
d_{t}:=\dim H^{t}(\omega_{K\langle T_{1},\ldots, T_{n}\rangle^{\dagger}/(T_{1}\cdots T_{m})}^{\bullet})
\end{equation*}
be the dimension of the log-Monsky-Washnitzer cohomology of $k[T_{1},\ldots, T_{n}]/(T_{1}\cdots T_{m})$.  Then $\Lambda_{n}$ is quasi-isomorphic to the complex (with zero differentials)
\begin{equation*}
A\xrightarrow{0}A^{d_{1}}\xrightarrow{0}A^{d_{2}}\xrightarrow{0}\cdots
\end{equation*}
\end{Lemma}

\begin{Lemma}\label{open}
With the same notation as above, let
\begin{equation*}
\breve{\Lambda}_{n}:=(A\otimes_{K}\omega_{\breve{D}^{n}(m)}^{0}\rightarrow A\otimes_{K}\omega_{\breve{D}^{n}(m)}^{1}\rightarrow A\otimes_{K}\omega_{\breve{D}^{n}(m)}^{2}\rightarrow\cdots)
\end{equation*}
be the analogous complex for $\breve{D}^{n}$ and its closed normal crossings divisor $\breve{D}^{n}(m)$. Then $\breve{\Lambda}_{n}$ is quasi-isomorphic to the  complex
\begin{equation*}
A\xrightarrow{0}A^{d_{1}}\xrightarrow{0}A^{d_{2}}\xrightarrow{0}\cdots
\end{equation*}
\end{Lemma}

We can now follow the  proof of (\cite{DLZ11}, Prop. 4.37). Let $\breve{D}^{n}=\displaystyle\bigcup_{i=1}^{\infty}U_{i}$ be a union of dagger balls of ascending radius, and let $\breve{D}^{n}(m)=\displaystyle\bigcup_{i=1}^{\infty}U_{i}(m)$ be the corresponding normal crossings divisors. For notational brevity, write $\omega^{q}:=\omega_{\tilde{Q}\times U_{i}(m)}^{q}$. Since $\tilde{Q}\times U_{i}(m)$ is affinoid, $H^{p}(\tilde{Q}\times U_{i}(m),\omega^{q})$ vanishes for $p\geq 1$ and $\mathbb{R}\Gamma(\tilde{Q}\times\breve{D}^{n}(m),\omega^{q})$ is quasi-isomorphic to the global sections of the complex
\begin{align*}
& \prod_{i=1}^{\infty}\omega^{q}(\tilde{Q}\times U_{i}(m))\rightarrow\prod_{i=1}^{\infty}\omega^{q}(\tilde{Q}\times U_{i}(m)) \\
& \ \ \ \ \ \ \ \ \  \ \ \prod s_{i}\mapsto\prod (s_{i}-s_{i+1})
\end{align*}
Note that $\omega_{\tilde{Q}\times U_{i}(m)}^{q}=\displaystyle\bigoplus_{\ell}\Omega_{\tilde{Q}}^{\ell}\otimes_{K}\omega_{U_{i}(m)}^{q-\ell}$. Then the complex $H^{0}(\tilde{Q}\times U_{i}(m),\omega_{\tilde{Q}\times U_{i}(m)}^{\bullet})$ is represented by the double complex with components 
\begin{equation*}
C^{p,q}(U_{t}(m))=H^{0}(\tilde{Q},\Omega_{\tilde{Q}}^{p})\otimes_{K}H^{0}(U_{i}(m),\omega^{q})
\end{equation*}
Therefore the morphism of double complexes 
\begin{align*}
 \prod_{i=1}^{\infty}C^{\bullet,\bullet}(U_{i}(m))\rightarrow\prod_{i=1}^{\infty}C^{\bullet,\bullet}(U_{i}(m))
\end{align*}
given on the $(p,q)$-entry by
\begin{align*}
& \prod_{i=1}^{\infty}C^{p,q}(U_{i}(m))\rightarrow\prod_{i=1}^{\infty}C^{p,q}(U_{i}(m)) \\
& \ \ \ \ \ \ \ \ \ \ \ \prod s_{i}\mapsto\prod (s_{i}-s_{i+1})
\end{align*}
induces a map of total complexes with kernel and cokernel $H^{0}(\tilde{Q}\times\breve{D}^{n}(m),\omega_{\tilde{Q}\times \breve{D}^{n}(m)}^{\bullet})$ and $H^{1}(\tilde{Q}\times\breve{D}^{n}(m),\omega_{\tilde{Q}\times\breve{D}^{n}(m)}^{\bullet})$, respectively. It follows from Lemma \ref{closed} that the total complex associated to $C^{\bullet,\bullet}(U_{i}(m))$ is quasi-isomorphic to 
\begin{equation*}
\bigoplus_{t}(H^{0}(\tilde{Q},\Omega_{\tilde{Q}}^{\bullet}))^{d_{t}}
\end{equation*}
with the correction $d_{0}=1$. Analogously, it follows from Lemma \ref{open} that $H^{0}(\tilde{Q}\times\breve{D}^{n}(m),\omega_{\tilde{Q}\times\breve{D}^{n}(m)}^{\bullet})$ is quasi-isomorphic to 
\begin{equation*}
\bigoplus_{t}(H^{0}(\tilde{Q},\Omega_{\tilde{Q}}^{\bullet}))^{d_{t}}=\left(\bigoplus_{t}(H^{0}(\tilde{Q},\Omega_{\tilde{Q}}^{0}))^{d_{t}}\rightarrow\bigoplus_{t}(H^{0}(\tilde{Q},\Omega_{\tilde{Q}}^{1}))^{d_{t}}\rightarrow\cdots\right)
\end{equation*}
Finally, $H^{1}(\tilde{Q}\times\breve{D}^{n}(m),\omega_{\tilde{Q}\times\breve{D}^{n}(m)}^{\bullet})$ is quasi-isomorphic to the total complex of the triple complex 
\begin{equation*}
H^{0}(\tilde{Q}\times\breve{D}^{n}(m),\Omega_{\tilde{Q}}^{\bullet}\otimes\omega_{\breve{D}^{n}(m)}^{\bullet})\rightarrow\prod_{i=1}^{\infty}C^{\bullet,\bullet}(U_{i}(m))\rightarrow\prod_{i=1}^{\infty}C^{\bullet,\bullet}(U_{i}(m))
\end{equation*}
which is quasi-isomorphic to the total complex of the double complex
\begin{align*}
\bigoplus_{t}(H^{0}(\tilde{Q},\Omega_{\tilde{Q}}^{\bullet}))^{d_{t}}\rightarrow\prod_{i=1}^{\infty}\bigoplus_{t}H^{0}(\tilde{Q}, 
& \Omega_{\tilde{Q}}^{\bullet})^{d_{t}}\rightarrow\prod_{i=1}^{\infty}\bigoplus_{t}H^{0}(\tilde{Q},\Omega_{\tilde{Q}}^{\bullet})^{d_{t}} \\
& \prod s_{i}\mapsto\prod (s_{i}-s_{i+1})
\end{align*}
(we note that the direct sums are finite since the $d_{t}=0$ for $t$ greater than twice the dimension).

Since the double complex is acyclic with regard to the horizontal differential, the total complex is acyclic too, and hence  $H^{1}(\tilde{Q}\times\breve{D}^{n}(m),\omega_{\tilde{Q}\times\breve{D}^{n}(m)}^{\bullet})$ is also acyclic. This proves Proposition 4.37 and Corollary 4.38 in \cite{DLZ11} for $\omega_{\tilde{Q}\times\breve{D}^{n}(m)}^{\bullet}$, and hence we conclude that Proposition \ref{crucial quasi-isomorphism} holds.

\section{The monodromy operator}

We follow the argument in \cite{Mok93} but in the more general setting that $Y$ need not be proper.

Let $Y$ be a quasi-projective semistable scheme over $S_{0}$. Define a double complex
\begin{equation*}
\mathcal{B}^{\dagger\bullet,\bullet}:=\frac{W^{\dagger}\tilde{\omega}_{Y/k}^{i+j+1}}{P_{j}W^{\dagger}\tilde{\omega}_{Y/k}^{i+j+1}}
\end{equation*}
with the differential $\mathcal{B}^{\dagger i,j}\rightarrow\mathcal{B}^{\dagger i+1,j}$ given by $(-1)^{j}d$ and the  differential $\mathcal{B}^{\dagger i,j}\rightarrow\mathcal{B}^{\dagger i,j+1}$ given by $\omega\mapsto\omega\wedge\theta$. Let $\mathcal{B}^{\dagger\bullet}$ be the total complex of $\mathcal{B}^{\dagger\bullet,\bullet}$. Then $\mathcal{B}^{\dagger\bullet}\otimes\mathbb{Q}$ is the complex $\mathcal{B}_{\mathbb{Q}}^{\dagger\bullet}$ considered in the proof of Theorem \ref{rational quasi-isomorphisms} in the log-Monsky-Washnitzer setting. Let $\Phi$ denote the map induced by $p^{i+1}F$ on $\mathcal{B}^{\dagger i,j}$. Define also a map $\nu$ by requiring that $(-1)^{i+j+1}\nu:\mathcal{B}^{\dagger i,j}\rightarrow\mathcal{B}^{\dagger i-1,j+1}$ is the projection. This induces a map on $\mathcal{B}^{\dagger\bullet}$, which we also call $\nu$. The same argument as in the proof of Theorem \ref{rational quasi-isomorphisms} shows that the natural map $W^{\dagger}\tilde{\omega}_{Y/k}^{\bullet}\rightarrow\mathcal{B}^{\dagger\bullet}$ factors through $\Theta:W^{\dagger}\omega_{Y/k}^{\bullet}\rightarrow\mathcal{B}^{\dagger\bullet}$, and $\Theta\Phi=\Phi\Theta$. One also has that $\Theta\otimes\mathbb{Q}$ is a quasi-isomorphism. Indeed, this is a local question on $Y$, so we may reduce to the case that $Y$ is a semistable affine scheme over $S_{0}$, and this case was already shown in the proof of Theorem \ref{rational quasi-isomorphisms}.

\begin{Proposition}
The map $\nu:\mathcal{B}^{\dagger\bullet,\bullet}\rightarrow\mathcal{B}^{\dagger\bullet,\bullet}$ induces a nilpotent operator $N$ on 
\begin{equation*}
\mathbb{H}^{\ast}(Y,\mathcal{B}^{\dagger\bullet}_{\mathbb{Q}})\cong\mathbb{H}^{\ast}(Y,W^{\dagger}\omega_{Y/k}^{\bullet}\otimes\mathbb{Q})\cong H_{\emph{\logrig}}^{\ast}(Y/\mathfrak{S}_{0})
\end{equation*}
which coincides with the monodromy operator
\begin{equation*}
N:H_{\emph{\logrig}}^{\ast}(Y/\mathfrak{S}_{0})\rightarrow H_{\emph{\logrig}}^{\ast}(Y/\mathfrak{S}_{0})
\end{equation*}
defined in (\cite{Gro05}, \S5.4).
\end{Proposition}
\begin{proof}
Let us define another double complex by
\begin{equation*}
\mathcal{C}^{\dagger i,j}:=\mathcal{B}^{\dagger, i-1,j}\oplus\mathcal{B}^{\dagger i,j}
\end{equation*}
for $i,j\geq 0$, with the differential $\mathcal{C}^{\dagger i,j}\rightarrow\mathcal{C}^{\dagger i+1,j}$ given by 
\begin{equation*}
(\omega_{1},\omega_{2})\mapsto((-1)^{j}d\omega_{1},(-1)^{j}d\omega_{2})
\end{equation*}
and the differential $\mathcal{C}^{\dagger i,j}\rightarrow\mathcal{C}^{\dagger i,j+1}$ given by 
\begin{equation*}
(\omega_{1},\omega_{2})\mapsto(\omega_{1}\wedge\theta+\nu\omega_{2},\omega_{2}\wedge\theta)
\end{equation*}
Let $\mathcal{C}^{\dagger\bullet}$ be the total complex of $\mathcal{C}^{\dagger\bullet,\bullet}$. Then we get a natural morphism 
\begin{equation*}
\Psi:W^{\dagger}\tilde{\omega}_{Y/k}^{\bullet}\rightarrow\mathcal{C}^{\dagger\bullet}
\end{equation*}
fitting into the following diagram of short exact sequences
\begin{equation*}
\begin{tikzpicture}[descr/.style={fill=white,inner sep=1.5pt}]
        \matrix (m) [
            matrix of math nodes,
            row sep=2.5em,
            column sep=2.5em,
            text height=1.5ex, text depth=0.25ex
        ]
        {0 & W^{\dagger}\omega_{Y/k}^{\bullet}[-1] & W^{\dagger}\tilde{\omega}_{Y/k}^{\bullet} & W^{\dagger}\omega_{Y/k}^{\bullet} & 0 \\
        0 & \mathcal{B}^{\dagger\bullet}[-1] & \mathcal{C}^{\dagger\bullet} & \mathcal{B}^{\dagger\bullet} & 0 \\
        };

        \path[overlay,->, font=\scriptsize]
       (m-1-1) edge (m-1-2)
       (m-1-2) edge node [above]{$-\wedge\theta$} (m-1-3)
       (m-1-3) edge (m-1-4)
       (m-1-4) edge (m-1-5)
       (m-2-1) edge (m-2-2)
       (m-2-2) edge node [above]{$-\wedge\theta$} (m-2-3)
       (m-2-3) edge (m-2-4)
       (m-2-4) edge (m-2-5)
       (m-1-2) edge node [right]{$\Theta[-1]$} (m-2-2)
       (m-1-3) edge node [right]{$\Psi$} (m-2-3)
       (m-1-4) edge node [right]{$\Theta$} (m-2-4);

       \end{tikzpicture} 
\end{equation*} 
Tensoring by $\mathbb{Q}$ gives the diagram
\begin{equation*}
\begin{tikzpicture}[descr/.style={fill=white,inner sep=1.5pt}]
        \matrix (m) [
            matrix of math nodes,
            row sep=2.5em,
            column sep=2.5em,
            text height=1.5ex, text depth=0.25ex
        ]
        {0 & W^{\dagger}\omega_{Y/k}^{\bullet}\otimes\mathbb{Q}[-1] & W^{\dagger}\tilde{\omega}_{Y/k}^{\bullet}\otimes\mathbb{Q} & W^{\dagger}\omega_{Y/k}^{\bullet}\otimes\mathbb{Q} & 0 \\
        0 & \mathcal{B}_{\mathbb{Q}}^{\dagger\bullet}[-1] & \mathcal{C}^{\dagger\bullet}\otimes\mathbb{Q} & \mathcal{B}_{\mathbb{Q}}^{\dagger\bullet} & 0 \\
        };

        \path[overlay,->, font=\scriptsize]
       (m-1-1) edge (m-1-2)
       (m-1-2) edge node [above]{$-\wedge\theta$} (m-1-3)
       (m-1-3) edge (m-1-4)
       (m-1-4) edge (m-1-5)
       (m-2-1) edge (m-2-2)
       (m-2-2) edge node [above]{$-\wedge\theta$} (m-2-3)
       (m-2-3) edge (m-2-4)
       (m-2-4) edge (m-2-5)
       (m-1-2) edge node [right]{$\Theta\otimes\mathbb{Q}[-1]$} node [left]{$\wr$} (m-2-2)
       (m-1-3) edge node [right]{$\Psi\otimes\mathbb{Q}$} (m-2-3)
       (m-1-4) edge node [right]{$\Theta\otimes\mathbb{Q}$} node [left]{$\wr$} (m-2-4);

       \end{tikzpicture} 
\end{equation*} 
where the outermost vertical arrows are quasi-isomorphisms by the local argument given in the proof of Theorem \ref{rational quasi-isomorphisms}, and hence we conclude that $\Psi\otimes\mathbb{Q}$ is also a quasi-isomorphism. By construction, this shows that the map 
\begin{equation*}
N:\mathbb{H}^{\ast}(Y,W^{\dagger}\omega_{Y/k}^{\bullet}\otimes\mathbb{Q})\rightarrow\mathbb{H}^{\ast}(Y,W^{\dagger}\omega_{Y/k}^{\bullet}\otimes\mathbb{Q})
\end{equation*}
induced by $\nu:\mathcal{B}^{\dagger\bullet,\bullet}\rightarrow\mathcal{B}^{\dagger\bullet,\bullet}$ is exactly the connecting homomorphism on cohomology associated to the top short exact sequence. 

It therefore suffices to prove that the connecting homomorphism gives the monodromy operator on $H_{\logrig}^{\ast}(Y/\mathfrak{S}_{0})$.
Let $Y_{\bullet}$ be the simplicial scheme and $M_{\bullet}:=]Y_{\bullet}[_{\hat{D}_{\bullet}}^{\dagger}$ the simplicial dagger space as constructed in the proof of Theorem \ref{main comparison}. Then we have a diagram of short exact sequences of complexes of simplicial sheaves
\begin{equation*}
\begin{tikzpicture}[descr/.style={fill=white,inner sep=1.5pt}]
        \matrix (m) [
            matrix of math nodes,
            row sep=2.5em,
            column sep=1.5em,
            text height=1.5ex, text depth=0.25ex
        ]
        {
        0 & \mathbb{R}\text{sp}_{\ast}\omega_{M_{\bullet}}^{\bullet}[-1] & \mathbb{R}\text{sp}_{\ast}\tilde{\omega}_{M_{\bullet}}^{\bullet} & \mathbb{R}\text{sp}_{\ast}\omega_{M_{\bullet}}^{\bullet} & 0 \\
        0 & W^{\dagger}\omega_{Y_{\bullet}/k}^{\bullet}\otimes\mathbb{Q}[-1] & W^{\dagger}\tilde{\omega}_{Y_{\bullet}/k}^{\bullet}\otimes\mathbb{Q} & W^{\dagger}\omega_{Y_{\bullet}/k}^{\bullet}\otimes\mathbb{Q} & 0 \\
        };

        \path[overlay,->, font=\scriptsize]
       (m-1-1) edge (m-1-2)
       (m-1-2) edge node [above]{$-\wedge\theta$} (m-1-3)
       (m-1-3) edge (m-1-4)
       (m-1-4) edge (m-1-5)
       (m-2-1) edge (m-2-2)
       (m-2-2) edge node [above]{$-\wedge\theta$} (m-2-3)
       (m-2-3) edge (m-2-4)
       (m-2-4) edge (m-2-5);
       
       \draw[shorten <=0.2cm,->, font=\scriptsize]
       (m-1-2) edge node [left]{$\wr$} (m-2-2)
       (m-1-3) edge (m-2-3)
       (m-1-4) edge node [left]{$\wr$} (m-2-4);

       \end{tikzpicture} 
\end{equation*} 
where the outermost vertical arrows are the quasi-isomorphisms (\ref{simplicial}), and the middle arrow is constructed as follows:

We rewrite the morphism (\ref{map 1}) in terms of dagger spaces for each $J$
\begin{equation*}
\Gamma(M_{J},\tilde{\omega}_{M_{J}}^{\bullet})\rightarrow W^{\dagger}\tilde{\omega}_{Y_{J}/k}^{\bullet}\otimes\mathbb{Q}
\end{equation*}
Then applying the argument after (\cite{DLZ11}, 4.32) gives a local version
\begin{equation*}
\text{sp}_{\ast}\tilde{\omega}_{M_{J}}^{\bullet}\rightarrow W^{\dagger}\tilde{\omega}_{Y_{J}/k}^{\bullet}\otimes\mathbb{Q}
\end{equation*}
The proof of Proposition \ref{crucial quasi-isomorphism} holds verbatim with $\text{sp}_{\ast}\omega_{M_{J}}^{\bullet}$ replaced by $\text{sp}_{\ast}\tilde{\omega}_{M_{J}}^{\bullet}$ to show that the canonical morphism $\text{sp}_{\ast}\tilde{\omega}_{M_{J}}^{\bullet}\rightarrow\mathbb{R}\text{sp}_{\ast}\tilde{\omega}_{M_{J}}^{\bullet}$ is a quasi-isomorphism. This then defines the middle arrow in the diagram, which is therefore a quasi-isomorphism. The monodromy operator on the log-rigid cohomology of $Y$ is, by definition, the connecting homomorphism on cohomology associated to the top short exact sequence.   

\section{Comparison with log-crystalline cohomology in the projective case}\label{projective section}

We prove a semistable analogue of a comparison, obtained for smooth projective varieties in \cite{LZ15} between overconvergent and usual de Rham-Witt cohomology, for Hyodo-Kato cohomology:

\begin{Theorem}\label{projective comparison}
Let $Y$ be a projective semistable scheme over $S_{0}$. Then the canonical map
\begin{equation*}
\mathbb{H}^{\ast}(Y,W^{\dagger}\omega_{Y/k}^{\bullet})\rightarrow \mathbb{H}^{\ast}(Y,W\omega_{Y/k}^{\bullet})
\end{equation*}
is an isomorphism of $W(k)$-modules of finite type.
\end{Theorem}
For the assumption of (quasi-)projectivity see the remark below Theorem \ref{intro theorem}.

First we need a lemma:

\begin{Lemma}
Under the assumptions of Theorem \ref{projective comparison}, there is a commutative diagram
\begin{equation*}
\begin{tikzpicture}[descr/.style={fill=white,inner sep=1.5pt}]
        \matrix (m) [
            matrix of math nodes,
            row sep=2.5em,
            column sep=2.5em,
            text height=1.5ex, text depth=0.25ex
        ]
        {
        \mathbb{H}^{\ast}(Y,W^{\dagger}\omega_{Y/k}^{\bullet}) & \mathbb{H}^{\ast}(Y,W\omega_{Y/k}^{\bullet}) \\
        \mathbb{H}^{\ast}(Y,W^{\dagger}\omega_{Y/k}^{\bullet}\otimes\mathbb{Q}) & \mathbb{H}^{\ast}(Y,W\omega_{Y/k}^{\bullet}\otimes\mathbb{Q}) \\
        H_{\emph{\logrig}}^{\ast}(Y/\mathfrak{S}_{0}) & H_{\emph{\logcris}}^{\ast}((Y,M)/(W(k),W(L)))\otimes\mathbb{Q} \\
        };

        \path[overlay,->, font=\scriptsize]
       (m-1-1) edge (m-1-2)
       (m-2-1) edge (m-2-2)
       (m-3-1) edge node [above] {$\sim$} (m-3-2)
       (m-1-1) edge  (m-2-1)
       (m-1-2) edge  (m-2-2)
       (m-3-1) edge node [left] {$\wr$} (m-2-1)
       (m-3-2) edge node [left] {$\wr$} (m-2-2);
       \end{tikzpicture} 
\end{equation*}
Here $M$ is the log structure on $Y$ given by $\mathcal{O}_{Y}\cap u_{\ast}\mathcal{O}_{U}^{\times}$ where $u:U\hookrightarrow Y$ is a smooth dense open, and $W(L)$ is the canonical lifting of the log structure $L$ on $\spec k$ given by $1\mapsto 0$ (previously denoted by $S_{0}$). 
\end{Lemma}
\begin{proof}
We need to show that the lower square commutes. The isomorphism on the left and right are the comparisons between log-rigid and overconvergent Hyodo-Kato, resp. between log-crystalline and Hyodo-Kato cohomology (\cite{HK94}, 4.19). These isomorphisms also hold if $Y$ is only quasi-projective. The lower horizontal isomorphism is the logarithmic analogue of a comparison between rigid and crystalline cohomology defined in \cite{Ber97} in the proof of Theorem 1.9. If there exists a global semistable frame the analogy with Berthelot's proof is clear, otherwise one has to proceed by simplicial methods. Using Grosse-Kl\"{o}nne's definition of log-rigid cohomology as the cohomology of simplicial dagger spaces (\cite{Gro05}, 1.5) one obtains a canonical map, by using $p$-adic formal schemes and rigid spaces instead of weak formal schemes and dagger spaces, to Shiho's analytic cohomology which is isomorphic to log-convergent cohomology by Shiho's Log Convergent Poincar\'{e} Lemma (\cite{Shi02}, Corollary 2.3.9). Using the isomorphism between log-convergent and log-crystalline cohomology (\cite{Shi02}, Theorem 3.1.1) one obtains the lower horizontal arrow for any semistable $Y$, not necessarily proper. If $Y$ is proper, then log-rigid cohomology is isomorphic to analytic resp. log-convergent cohomology by (\cite{Gro05}, Theorem 5.3) and hence the lower horizontal arrow is an isomorphism for $Y$ proper semistable.

Hence all maps in the lower square are defined for quasi-projective varieties as well. Using the Mayer-Vietoris sequence for cohomology, we may assume that $Y$ is affine. Since the lower horizontal map in the diagram is independent from the choice of embeddings into log-smooth (weak) formal schemes, we may assume that $H^{\ast}_{\logrig}(Y/\mathfrak{S}_{0})$ is given by logarithmic Monsky-Washnitzer cohomology $H_{\logmw}^{\ast}(Y/K)$. In this case the map is given by a morphism of complexes
\begin{equation*}
\omega_{\tilde{Y}^{\dagger}}^{\bullet}\rightarrow\omega_{\hat{\tilde{Y}}}^{\bullet}
\end{equation*}
i.e. by taking $p$-adic completion of the logarithmic Monsky-Washnitzer complex. The comparison maps to the overconvergent and usual Hyodo-Kato complexes evidently commute with taking $p$-adic completions. This proves the lemma.
\end{proof}

Next we show the analogue of (\cite{LZ15}, 2.2).

\begin{Proposition}\label{reducing mod p^n}
Under the assumptions of Theorem \ref{projective comparison}, we have quasi-isomorphisms
\begin{equation*}
W^{\dagger}\omega_{Y/k}^{\bullet}/p^{n}\cong W\omega_{Y/k}^{\bullet}/p^{n}\cong W_{n}\omega_{Y/k}^{\bullet}
\end{equation*}
for all $n\in\mathbb{N}$.
\end{Proposition}

Only the first quasi-isomorphism requires a proof, the second quasi-isomorphism follows from (\cite{HK94}, Cor. 4.5).

This is a Zariski-local question, so we may assume that $Y$ is affine. Moreover, by a result of Kedlaya (\cite{Ked05}, Thm. 2), we may assume that $Y=\spec B$ is finite \'{e}tale and free over $\spec A=\spec k[T_{1},\ldots, T_{d}]/(T_{1}\cdots T_{r})$ for some $r$.

We note that (\cite{LZ15}, Prop. 2.3) is based on (\cite{DLZ12}, Cor. 2.46) and does not need that $A$ is a smooth $k$-algebra, hence we conlude that $W^{\dagger}(B)$ is a finite \'{e}tale $W^{\dagger}(A)$-algebra and free as a $W^{\dagger}(A)$-module.

The proof of (\cite{DLZ11}, Prop. 1.9) transfers verbatim to the Hyodo-Kato complexes and extends the \'{e}tale base change  for the Hyodo-Kato complexes in (\cite{Mat17}, Prop. 3.7) to the overconvergent setting, hence we have
\begin{equation*}
W^{\dagger}\omega_{A/k}^{\ell}\otimes_{W^{\dagger}(A)}W^{\dagger}(B)\xrightarrow{\sim}W^{\dagger}\omega_{B/k}^{\ell}
\end{equation*}

Let $\kappa_{A}:\tilde{A}^{\dagger}=W(k)\langle T_{1},\ldots, T_{d}\rangle^{\dagger}/(T_{1}\cdots T_{r})\rightarrow W^{\dagger}(A)$ be the canonical map obtained by sending $T_{i}$ to $[T_{i}]$ for $i=1,\ldots, d$. Note that $[T_{1}\cdots T_{r}]=[T_{1}]\cdots[T_{r}]$ is zero in $W(A)$, hence $\kappa_{A}$ is well-defined. By reproducing the argument before (\cite{LZ15}, Prop. 2.5), we conclude that the above map extends uniquely to 
\begin{equation*}
\kappa_{B}:\tilde{B}^{\dagger}\rightarrow W^{\dagger}(B)
\end{equation*}
(note that this map is used to construct the comparison morphisms (\ref{comparison tilde}) and (\ref{comparison})). Then we have

\begin{Proposition}
Let $B$ be finite \'{e}tale and free over $A=k[T_{1},\ldots, T_{d}]/(T_{1}\cdots T_{r})$. Then there is a decomposition of $W^{\dagger}\omega_{B/k}^{\bullet}$ into subcomplexes
\begin{equation*}
W^{\dagger}\omega_{B/k}^{\bullet}=W^{\dagger}\omega_{B/k}^{\text{int }\bullet}\oplus W^{\dagger}\omega_{B/k}^{\text{frac }\bullet}
\end{equation*}
where $W^{\dagger}\omega_{B/k}^{\text{frac }\bullet}$ is acyclic and $W^{\dagger}\omega_{B/k}^{\text{int }\bullet}$ is isomorphic to $\omega_{\tilde{B}^{\dagger}}^{\bullet}$ via the morphism induced by
\begin{equation*}
\kappa_{B}:\omega_{\tilde{B}^{\dagger}}^{\bullet}\rightarrow W^{\dagger}\omega_{B/k}^{\bullet}
\end{equation*}
\end{Proposition}
\begin{proof}
It is enough to treat the case $A=B$; the argument for this is the same as in the proof of (\cite{LZ15}, Prop. 2.5). Indeed, 
\begin{equation*}
W^{\dagger}\omega_{B/k}^{\ell}\cong W^{\dagger}\omega_{A/k}^{\ell}\otimes_{W^{\dagger}(A)}W^{\dagger}(B)
\end{equation*}
by \'{e}tale base change. Let $b_{1},\ldots, b_{m}$ be an $A$-module basis of $B$ and lift these to an $\tilde{A}^{\dagger}$-module basis $u_{1},\ldots, u_{m}$ of $\tilde{B}^{\dagger}$. Then $\kappa_{B}(u_{1}),\ldots,\kappa_{B}(u_{m})$ is a $W^{\dagger}(A)$-module basis of $W^{\dagger}(B)$. Therefore
\begin{equation*}
W^{\dagger}\omega_{B/k}^{\ell}\cong W^{\dagger}\omega_{A/k}^{\ell}\otimes_{\tilde{A}^{\dagger}}\tilde{B}^{\dagger}
\end{equation*}
If $W^{\dagger}\omega_{A/k}^{\bullet}$ decomposes as in the statement of the proposition, then we obtain a decomposition 
\begin{equation*}
W^{\dagger}\omega_{B/k}^{\bullet}=W^{\dagger}\omega_{B/k}^{\text{int }\bullet}\oplus W^{\dagger}\omega_{B/k}^{\text{frac }\bullet}
\end{equation*}
by tensoring the corresponding decomposition for $A$ with $\tilde{B}^{\dagger}$. The same proof as (\cite{DLZ11}, Theorem 3.19) shows that if $W^{\dagger}\omega_{A/k}^{\text{frac }\bullet}$ is acyclic then $W^{\dagger}\omega_{B/k}^{\text{frac }\bullet}$ is acyclic.
It therefore suffices to prove the proposition for the case $A=B$. For this we use the description of the de Rham-Witt complex of a (Laurent-) polynomial algebra given in (\cite{BMS16}, \S10.4):

For a $\mathbb{Z}_{(p)}$-algebra $R$, any element $\omega$ in $W_{n}\Omega_{R[T_{1}^{\pm 1},\ldots, T_{d}^{\pm 1}]/R}^{\ell}$ can be uniquely written as a finite sum
\begin{equation}\label{basic Witt}
\omega=\sum_{k,\mathcal{P}}e(\xi_{k,\mathcal{P}},k,\mathcal{P}_{\leq \rho})\prod_{j=\rho+1}^{\ell}d\log(\prod_{i\in I_{j}}[T_{i}])
\end{equation}
where $k$ ranges over the weight functions $k:[1,d]\rightarrow\mathbb{Z}[\frac{1}{p}]\cup\{\infty\}$ satisfying properties (i), (ii), (iii) in (\cite{BMS16},\S10.4), and $\mathcal{P}=\{I_{0},I_{1},\ldots,I_{\rho},I_{\rho+1},\ldots I_{\ell}\}$ is a disjoint partition of $I=\text{supp }k$, such that $\mathcal{P}_{\leq\rho}=\{I_{0},I_{1},\ldots, I_{\rho}\}$ and $\rho$ is the integer denoted by $\rho_{2}$ in (\cite{BMS16},\S10.4), $I_{0}$ is possibly empty and $e(\xi_{k,\mathcal{P}},k,\mathcal{P}_{\leq\rho})$ is a basic Witt differential of type Case 1, Case 2, Case 3 given in (\cite{LZ04}, 2.15-2.17) (but where the exponents of the $T_{i}$ for $i$ occurring in $I_{j}$ for $0\leq j\leq\rho$ can be negative).

Consider now the log-scheme $\specno(A,\mathbb{N}^{r})$ where $\mathbb{N}^{r}\ni e_{i}\mapsto T_{i}$, $1\leq i\leq r$, over the trivial base $\specno(k,\ast)$. Then the complex $W\Lambda_{(A,\mathbb{N}^{r})/(k,\ast)}^{\bullet}$, defined in \cite{Mat17}, can be described as follows (our description differs from the description in \cite{Mat17} but is equivalent): any $\omega$ in $W\Lambda_{(A,\mathbb{N}^{r})/(k,\ast)}^{\ell}$ has a unique expression as a convergent sum
\begin{equation}\label{convergent sum}
\omega=\sum_{k,\mathcal{P}}e(\xi_{k,\mathcal{P}},k,\mathcal{P}_{\leq \rho})\prod_{j=\rho+1}^{\ell}d\log(\prod_{i\in I_{j}}[T_{i}])
\end{equation}
as in (\ref{basic Witt}), where for any given $m$ we have
\par
\ \ \ - $\xi_{k,\mathcal{P}}\in V^{m}W(k)=p^{m}W(k)$ for all but finitely many $k$.
\par
\ \ \ - All weight functions take non-negative values, i.e. on $I_{0}\cup\cdots\cup I_{\rho}$ they\par \ \ \ \ \ take values in $\mathbb{Z}_{\geq 0}[\frac{1}{p}]$.
\par
\ \ \ - $[1,r]\not\subset I_{j}$ for any $j=0,\ldots,\rho$, and for any $i$ occuring in $I_{j}$ for
\par \ \ \ \ \ $j=\rho+1,\ldots,\ell$ we have $i\in[1,\ldots,r]$.
\\
It is clear from this description that we get a decomposition
\begin{equation*}
W\Lambda_{(A,\mathbb{N}^{r})/(k,\ast)}^{\bullet}=W\Lambda_{(A,\mathbb{N}^{r})/(k,\ast)}^{\text{int }\bullet}\oplus W\Lambda_{(A,\mathbb{N}^{r})/(k,\ast)}^{\text{frac }\bullet}
\end{equation*}
given by integral and purely fractional weights, and that the fractional part is acyclic, as in the case of (Laurent-) polynomial algebras (\cite{BMS16}, Theorem 10.13).

Now we apply (\cite{Mat17},\S7.2). Let $W_{m}\tilde{\Lambda}^{\bullet}:=W_{m}\Lambda_{(A,\mathbb{N}^{r})/(k,\ast)}^{\bullet}$ and $W_{m}\Lambda^{\bullet}:=W_{m}\Lambda_{(A,\mathbb{N}^{r})/(k,\mathbb{N})}^{\bullet}=W_{m}\Lambda_{(A,\mathbb{N}^{r})/S_{0}}^{\bullet}$, which is isomorphic to the Hyodo-Kato complex $W_{m}\omega_{Y/k}^{\bullet}$ by the proof of Proposition \ref{old-modern}. Then we have a short exact sequence (\cite{Mat17}, Lemma 7.4)
\begin{equation*}
0\longrightarrow W_{m}\Lambda^{\bullet-1}\xrightarrow{\wedge\theta_{m}}W_{m}\tilde{\Lambda}^{\bullet}\longrightarrow W_{m}\Lambda^{\bullet}\longrightarrow 0
\end{equation*}
where $\theta_{m}:=d\log[T_{1}]+\cdots+d\log[T_{r}]$. This implies that any element $\omega$ in $W_{m}\Lambda^{\ell}$ can be written uniquely as a sum
\begin{equation}\label{unique expression}
\omega=\sum_{k,\mathcal{P}}e(\xi_{k,\mathcal{P}},k,\mathcal{P}_{\leq \rho})\prod_{j=\rho+1}^{\ell}d\log(\prod_{i\in I_{j}}[T_{i}])
\end{equation}
with the following properties:
\par
\ \ \ - $[1,r]\not\subset I_{j}$ for any $j=0,\ldots,\ell$; $\rho$ is equal to $\rho_{2}$ in (\cite{BMS16},\S10.4).
\par
\ \ \ - For all $j=\rho+1,\ldots,\ell$ we have $I_{j}\subset\{1,\ldots,r\}$.
\par
\ \ \ - $e(\xi_{k,\mathcal{P}},k,\mathcal{P}_{\leq\rho})$ as before.
\\
From the definitions it is clear that we again have a decomposition
\begin{equation*}
W_{m}\Lambda^{\bullet}=W_{m}\Lambda^{\text{int }\bullet}\oplus W_{m}\Lambda^{\text{frac }\bullet}
\end{equation*}
and the fractional part is acyclic. Passing to the projective limit and overconvergent subcomplexes, we obtain decompositions
\begin{equation*}
W\Lambda^{\bullet}=W\Lambda^{\text{int }\bullet}\oplus W\Lambda^{\text{frac }\bullet}
\end{equation*}
and
\begin{equation*}
W^{\dagger}\Lambda^{\bullet}=W^{\dagger}\Lambda^{\text{int }\bullet}\oplus W^{\dagger}\Lambda^{\text{frac }\bullet}
\end{equation*}
and the fractional parts are acyclic subcomplexes (the acyclicity is inherited from the case of polynomial algebras). Hence we have the desired decomposition
\begin{equation*}
W^{\dagger}\omega_{Y/k}^{\bullet}=W^{\dagger}\omega_{Y/k}^{\text{int }\bullet}\oplus W^{\dagger}\omega_{Y/k}^{\text{frac }\bullet}
\end{equation*}
in the case that $Y=\spec k[T_{1},\ldots, T_{d}]/(T_{1}\cdots T_{r})$. It is evident that $W^{\dagger}\omega_{Y/k}^{\text{int }\bullet}$ is  isomorphic to $\omega_{\tilde{A}^{\dagger}}^{\bullet}$.
\end{proof}

Since the $W^{\dagger}\omega_{Y/k}^{\ell}$ and $W\omega_{Y/k}^{\ell}$ are $p$-torsion free (\cite{HK94}, Cor. 4.5), we conclude that $W^{\dagger}\omega_{Y/k}^{\text{frac }\bullet}\otimes\mathbb{Z}/p^{n}$ is acyclic too. It is clear that $\omega_{\tilde{B}^{\dagger}}^{\bullet}\otimes\mathbb{Z}/p^{n}$ is isomorphic to $\omega_{\hat{\tilde{B}}}^{\bullet}\otimes\mathbb{Z}/p^{n}$. This concludes the proof of Proposition \ref{reducing mod p^n}.

Finally, since 
\begin{equation*}
\varprojlim_{n}\mathbb{H}^{i}(Y,W^{\dagger}\omega_{Y/k}^{\bullet}/p^{n})=\varprojlim_{n}\mathbb{H}^{i}(Y,W_{n}\omega_{Y/k}^{\bullet})=\mathbb{H}^{i}(Y,W\omega_{Y/k}^{\bullet})
\end{equation*}
where the last equality holds because all $\mathbb{H}^{i}(Y,W_{n}\omega_{Y/k}^{\bullet})$ are of finite length if $Y$ is proper (\cite{HK94}, \S3.2) and $\mathbb{H}^{i}(Y,W\omega_{Y/k}^{\bullet})$ are $W(k)$-modules of finite type, we can apply the arguments in (\cite{LZ15}, page 1392) to conclude that
\begin{equation*}
\mathbb{H}^{\ast}(Y,W^{\dagger}\omega_{Y/k}^{\bullet})\cong\mathbb{H}^{\ast}(Y,W\omega_{Y/k}^{\bullet})
\end{equation*}
This proves Theorem \ref{projective comparison}.

\end{proof}

\
\\
\
\\
\noindent 
Oli Gregory \\
Technische Universit\"{a}t M\"{u}nchen \\
Zentrum Mathematik - M11 \\
Boltzmannstra{\ss}e 3 \\
85748 Garching bei M\"{u}nchen, Germany \\
email: oli.gregory@tum.de
\
\\
\
\\
Andreas Langer\\
University of Exeter\\
Mathematics\\
Exeter EX4 4QF\\
Devon, UK\\
email: a.langer@exeter.ac.uk


\begin{thebibliography}{9}

\bibitem[Ber97]{Ber97}
P. Berthelot, \emph{Finitude et puret\'{e} cohomologique en cohomologie rigide}, Invent. Math. {\bf{128}} (1997), 329-377.

\bibitem[BMS16]{BMS16}
B. Bhatt, M. Morrow, P. Scholze, \emph{Integral $p$-adic Hodge theory}, preprint (2016), arXiv:1602.03148 [math.AG].

\bibitem[CDN18]{CDN18}
P. Colmez, G. Dospinescu, W. Niziol, \emph{Cohomology of $p$-adic Stein spaces}, preprint (2018), arXiv:1801.06686 [math.NT].

\bibitem[DLZ11]{DLZ11}  C. Davis, A. Langer, T. Zink, \emph{Overconvergent de Rham-Witt Cohomology}, Annals. Sc. Ec. Norm. Sup. {\bf{44}}, no. 2 (2011), 197-262.

\bibitem[DLZ12]{DLZ12}  C. Davis, A. Langer, T. Zink, \emph{Overconvergent Witt vectors}, J. Reine. Angew. Math. {\bf{668}} (2012), 1-34.

\bibitem[deS05]{deS05}
E. de Shalit, \emph{The $p$-adic weight-monodromy conjecture for $p$-adically uniformized varieties}, Compositio Math. {\bf{141}} (2005), 101-120.

\bibitem[Elk73]{Elk73}
R. Elkik, \emph{Solutions d'\'{e}quations a coefficients dans un anneau heselien}, Annals. Sc. Ec. Norm. Sup. {\bf{6}}, no. 4 (1973), 553-604.

\bibitem[Gro00]{Gro00}
E. Grosse-Kl\"{o}nne, \emph{Rigid analytic spaces with overconvergent structure sheaf}, J. reine angew. Math. {\bf{519}} (2000), 73-95.

\bibitem[Gro05]{Gro05}
E. Grosse-Kl\"{o}nne, \emph{Frobenius and Monodromy operators in rigid analysis, and Drinfel'd's symmetric space}, J. Algebraic Geometry {\bf{14}}, no.3 (2005), 391-437.

\bibitem[Hyo91]{Hyo91}
O. Hyodo, \emph{On the de Rham-Witt complex attached to a semi-stable family}, Comp. Math. {\bf{78}}, no. 3 (1991), 241-260. 

\bibitem[HK94]{HK94}
O. Hyodo, K. Kato, \emph{Semi-stable reduction and crystalline cohomology with logarithmic poles}, Ast\'{e}risque {\bf{223}} (1994), 209-220.

\bibitem[Ill79]{Ill79}
L. Illusie, \emph{Complexe de de Rham-Witt et cohomologie cristalline}, Ann. Sci. \'{E}cole. Norm. Sup. (4) {\bf{12}} (1979), 501-661.

\bibitem[Kat96]{Kat96}
F. Kato, \emph{Log smooth deformation theory}, Tohoku Math. J. {\bf{48}}, no. 3 (1996), 317-354. 

\bibitem[Ked05]{Ked05}
K. S. Kedlaya, \emph{More \'{e}tale covers of affine spaces in positive characteristic}, J. Algebraic Geom. {\bf{14}} (2005),  187-192.

\bibitem[LZ04]{LZ04}
A. Langer, T. Zink, \emph{De Rham-Witt cohomology of a proper and smooth morphism}, J. Inst. Math. Jussieu {\bf{3}} (2004), 231-314.

\bibitem[LZ15]{LZ15}
A. Langer, T. Zink, \emph{Comparison between overconvergent de Rham-Witt and crystalline cohomology for projective smooth varieties}, Math. Nachrichten {\bf{288}} 11-12 (2015), 1388-1393.

\bibitem[Law18]{Law18}
N. Lawless, \emph{A Comparison of Overconvergent Witt de-Rham Cohomology and Rigid Cohomology on Smooth Schemes}, preprint (2018), arXiv:1810.10059 [math.NT].


\bibitem[Mat17]{Mat17}
H. Matsuue, \emph{On relative and overconvergent de Rham-Witt cohomology for log schemes}, Math. Zeit. {\bf{286}}, no. 1 (2017), 19-87.

\bibitem[Mok93]{Mok93}
A. Mokrane, \emph{La suite spectrale des poids en cohomologie de Hyodo-Kato}, Duke Math. J. {\bf{72}} (1993), 301-337. 

\bibitem[Shi02]{Shi02}
A. Shiho, \emph{Crystalline Fundamental Groups II - Log Convergent Cohomology and Rigid Cohomology}, J. Math. Sci. Univ. Tokyo {\bf{9}}, no.1 (2002), 1-163  

\bibitem[Ste76]{Ste76}
J. Steenbrink, \emph{Limits of Hodge Structures}, Invent. Math. {\bf{31}} (1976), 229-257

\end{thebibliography}
\end{document}